\documentclass[reqno]{amsart}

\usepackage{amsmath, amsthm, amssymb, latexsym, amsthm, url, color}
\definecolor{medblue}{rgb}{0.09,0.32,0.44} %22-84-113
\usepackage[pdfborder={0 0 0},pdfborderstyle={},colorlinks=true,linkcolor=medblue,citecolor=medblue,urlcolor=blue]{hyperref}

\def\clap#1{\hbox to 0pt{\hss#1\hss}}
\allowdisplaybreaks

\newtheorem{thm}{Theorem}[section]
\newcommand{\bthm}{\begin{thm}} \newcommand{\ethm}{\end{thm}}
\newtheorem{prop}[thm]{Proposition}
\newcommand{\bprp}{\begin{prop}} \newcommand{\eprp}{\end{prop}}
\newtheorem{fact}[thm]{Fact}
\newcommand{\bfct}{\begin{fact}} \newcommand{\efct}{\end{fact}}
\newtheorem{prob}[thm]{Problem}
\newcommand{\bprb}{\begin{prob}} \newcommand{\eprb}{\end{prob}}
\newtheorem{lem}[thm]{Lemma}
\newcommand{\blem}{\begin{lem}} \newcommand{\elem}{\end{lem}}
\newtheorem{claim}[thm]{Claim}
\newcommand{\bclm}{\begin{claim}} \newcommand{\eclm}{\end{claim}}
\newtheorem{corollary}[thm]{Corollary}
\newcommand{\bcor}{\begin{corollary}} \newcommand{\ecor}{\end{corollary}}
\newtheorem{conj}[thm]{Conjecture}
\newcommand{\bcnj}{\begin{conj}} \newcommand{\ecnj}{\end{conj}}
\theoremstyle{definition}
\newtheorem{definition}[thm]{Definition}
\newcommand{\bdfn}{\begin{definition}} \newcommand{\edfn}{\end{definition}}
\theoremstyle{remark}
\newtheorem{rem}[thm]{Remark}
\newcommand{\brem}{\begin{rem}} \newcommand{\erem}{\end{rem}}
\newtheorem{cnv}[thm]{Convention}
\newcommand{\bcnv}{\begin{cnv}} \newcommand{\ecnv}{\end{cnv}}
\newtheorem{exam}[thm]{Example}
\newcommand{\bexm}{\begin{exam}} \newcommand{\eexm}{\end{exam}}
\newcommand{\bpf}{\begin{proof}} \newcommand{\epf}{\end{proof}}
\newtheorem{exer}[thm]{Exercise}
\newcommand{\bexer}{\begin{exer}} \newcommand{\eexer}{\end{exer}}

\newcommand{\R}{{\mathbb R}}
\newcommand{\Z}{{\mathbb Z}}
\newcommand{\N}{{\mathbb N}}

\renewcommand{\Pr}{\mathbb{P}}
\newcommand{\E}{\mathbb{E}}
\newcommand{\Var}{\textrm{Var}}
\renewcommand{\a}{\alpha}
\newcommand{\om}{\omega}
\newcommand{\Norm}{\mathcal{N}}
\newcommand{\eps}{\varepsilon}

\newcommand{\ov}{\overline}
\newcommand{\half}{\frac{1}{2}}
\newcommand{\floor}[1]{{\lfloor#1\rfloor}}
\newcommand{\ceil}[1]{{\lceil#1\rceil}}
\newcommand\ip[1]{\langle{#1}\rangle}

\title{Excited random walk with periodic cookies}

\author{
Gady Kozma
\and
Tal Orenshtein
\and
Igor Shinkar
}
\begin{document}

\begin{abstract}
    In this paper we consider an excited random walk (ERW) on $\Z$ in
    identically piled periodic environment. This is a discrete time process
    on $\Z$ defined by parameters $(p_1,\dots p_M) \in [0,1]^M$ for some
    positive integer $M$, where the walker upon the $i^\textrm{th}$ visit to $z \in \Z$
    moves to $z+1$ with probability $p_{i\pmod M}$, and moves to $z-1$ with
    probability $1-p_{i \pmod M}$.
    We give an explicit formula in terms of the parameters $(p_1,\dots,p_M)$
    which determines whether the walk is recurrent, transient to the left,
    or transient to the right.
    In particular, in the case that
$\frac{1}{M}\sum_{i=1}^{M}p_{i}=\frac 12$ all behaviors are possible,
    and may depend on the order of the $p_i$.
Our framework allows us to reprove some known results on ERW and branching processes with migration with no
additional effort.
\end{abstract}

\maketitle



\newpage

\hypersetup{linkcolor=black}
\tableofcontents
\hypersetup{linkcolor=medblue}

\newpage
%%%%%%%%%%%%%%%%%%%%%%%%%%%%%%%%%%%%%%%%%%%%%%%%%%%%%%%%%%%%%%%%%%%%%%%%%%%%%%%
\section{Introduction}\label{sec:intro}
%%%%%%%%%%%%%%%%%%%%%%%%%%%%%%%%%%%%%%%%%%%%%%%%%%%%%%%%%%%%%%%%%%%%%%%%%%%%%%%

Excited Random Walk (ERW), also known as the Cookie Walk, was
introduced by Benjamini and Wilson~\cite{benjamini2003excited} as a non
Markovian local perturbation of simple random walk on $\Z^d$, $d\ge 1$.
In this model we have a stack of cookies placed on each vertex of the
lattice, and each cookie encodes a probability distribution on the next
step of the walker (also known as the cookie monster). In each round the walker eats
the top cookie in the stack in her current position, and makes a random
step according to the probability distribution encoded by this cookie.
In their paper Benjamini and Wilson~\cite{benjamini2003excited} showed
that by placing a single biased cookie in each vertex the walk is recurrent
in $\Z$ and is transient in $\Z^d$ for all $d \geq 2$. The case $d=1$
has been later generalized by Zerner~\cite{zerner2005multi} by
placing more biased cookies in each vertex of the lattice. There has
been a lot of work done on this model, in both deterministic and
random cookie environments. For more background see the recent survey
of Kosygina and Zerner \cite{kosygina2012excited} and the references therein.

In this paper we shall discuss only the case $d=1$. ERW is a discrete
time stochastic process $X=(X_n)_{n \geq 0}$ on the integer lattice $\Z$.
The process $X$ in the cookie environment $\om \in [0,1]^{\Z \times \N}$ is
initiated at some $X_0 = x \in \Z$. If at time $n$ the walker is in
position $y$, and this is her $j^\textrm{th}$ visit to $y$, then she moves to $y+1$
with probability $\om(y,j)$, and moves to $y-1$ with probability $1-\om(y,j)$.

%Formally, it is defined by the probability measure $\Pr_{x,\om}$,
%where $\Pr_{x,\om}(X_0=x) = 1$, and for $n \geq 0$ it holds that
%\[
%\Pr_{x,\om}[X_{n+1} = X_{n}+1\ |\ X_{0},X_{1},\dots,X_{n}]
%=
%\om(X_{n},\#\{k\le n:X_{k}=X_{n}\})
%\]
%and
%\[
%\Pr_{x,\om}[X_{n+1}=X_{n}-1\ |\ X_{0},X_{1},\dots,X_{n}]
%=
% 1 - \Pr_{x,\om}[X_{n+1} = X_{n}+1\ |\ X_{0},X_{1},\dots,X_{n}].
%\]
%We usually omit the subscript in $\Pr_{x,\om}$,
%and write $\Pr$ instead when it is clear from context.

In this paper we shall assume that the initial position $x$ of the process
is $0$, and that the stacks in the cookie environment $\om$ are
\emph{identically piled}, that is $\om(x,i)=\om(0,i)$ for all $x \in \Z$
and $i \in \N$. For a vector $p \in [0,1]^{\N}$ we shall write $\om(p)$ to
denote the identically piled cookie environment $\om$ where $\om(x,i)=p_i$
for all $x \in \Z$ and $i \in \N$.
In this case the ERW mechanism has a simple form, namely
$\Pr[X_0 = 0] = 1$ and
\[
    \Pr[X_{n+1} = X_n+1\ |\ L_n(X_n)=j] = 1-\Pr[X_{n+1} = X_n-1\ |\ L_n(X_n)=j] = p_j,
\]
where $L_n(x)=\#\{k \leq n : X_k=x \}$ is the number of visits to $x\in\Z$ in $n$ steps of the walk.
With a minor abuse of notation we say that $p$ is a cookie environment,
when actually referring to the identically piled cookie environment $\om(p)$.
Next, we introduce some definitions for cookie environments.
\begin{definition}
    Let $p \in [0,1]^{\N}$ be a cookie environment.
    \begin{itemize}
    \item
    The environment $p$ is called \emph{elliptic} if $p\in (0,1)^{\N}$.
    \item
    The environment $p$ is called \emph{non-degenerate}
    if $\sum_{i=1}^\infty p_i=\sum_{i=1}^\infty (1-p_i)=\infty$
    \item
    The environment $p$ is called \emph{positive} if $p\in[\half,1]^\N$.
    \item
    The environment $p$ is called \emph{bounded} if there is some $M \in \N$
    such that $p_i=\half$ for all $i>M$
    \item
    The environment $p$ is called \emph{periodic} if for some $M \in \N$
    it holds that $p_i=p_{i+M}$ for all $i \in \N$.
    We denote such an environment by $\om(p_1,\dots,p_M)$.
    \end{itemize}
\end{definition}

%%%%%%%%%%%%%%%%%%%%%%%%%%%%%%%%%%%%%%%%%%%%%%%%%%%%%%%%%%%%%%%%%%%%%%%%%%%%%%%
\subsection{Our results}\label{subsec:our results}
%%%%%%%%%%%%%%%%%%%%%%%%%%%%%%%%%%%%%%%%%%%%%%%%%%%%%%%%%%%%%%%%%%%%%%%%%%%%%%%

Consider an ERW in an elliptic periodic environment. That is, the environment is defined by parameters
$(p_1,\dots p_M) \in (0,1)^M$ for some $M \in \N$, where the walker upon
the $i^\textrm{th}$ visit to $z \in \Z$ moves to $z+1$ with probability
$p_{i \pmod M}$, and moves to $z-1$ with probability $1-p_{i \pmod M}$
(where we identify $p_0$ with $p_M$). Our main result is the following.

\begin{thm}[Transience criterion for periodic environments]
\label{thm:periodic env}
    Let $(p_1,\dots,p_M) \in (0,1)^M$ for some $M \in \N$, and let
    $\ov{p} = \frac{1}{M}\sum_{i=1}^M p_i$ be the average of the $p_i$'s.
    Let $X = (X_n)_{n \geq 0}$ be a ERW in the periodic environment
    $\om = \om(p_1,\dots,p_M)$.
    \begin{enumerate}
    \item
    If $\ov{p}>\half$, then $\Pr$-a.s. $X_n \to +\infty$ as $n \to \infty$.
    \item
    If $\ov{p}<\half$ then $\Pr$-a.s. $X_n \to -\infty$ as $n \to \infty$.
    \item
    Suppose that $\ov{p} = \half$, and let
    \begin{equation}\label{eq:periodic theta}
        \theta(p_1,\dots,p_M)=
        \frac{\sum_{i=1}^M \delta_i (1-p_{i})}
            {4 \sum_{j=1}^M p_j(1-p_j)},
    \end{equation} where $\delta_i=\sum_{j=1}^{i} (2p_j - 1)$.
        \begin{itemize}
         \item  If $\theta(p_1,\dots,p_M)>1$, then $\Pr$-a.s. $X_n\to+\infty$.
         \item  If $\theta(1-p_1,\dots,1-p_M)>1$, then $\Pr$-a.s. $X_n\to-\infty$.
         \item  If both $\theta(p_1,\dots,p_M) \leq 1$ and $\theta(1-p_1,\dots,1-p_M) \leq 1$,
         then $\Pr$-a.s. $X_n=0 \mbox{ infinitely often.}$
        \end{itemize}
    \end{enumerate}
\end{thm}

It is interesting to compare Theorem~\ref{thm:periodic env} to results
about the case of bounded environment. Recall~\cite{kosygina2008positively} that in the case of bounded
environment the only value that matters is the \emph{total drift},
i.e., the sum $\sum_i(2p_i-1)$. If the total drift is bigger than 1, then the
walk is transient to the right, if it is smaller than $-1$ then the walk
is transient to the left, and if it is in $[-1,1]$ then the walk is
recurrent. For other phase transitions in the total drift see
\cite{kosygina2013phase}. Comparing to the bounded case, the cases
$\ov{p}>\frac12$ and $\ov{p}<\frac12$ (which corresponds to total
drift infinite and negative infinite, respectively) are not
surprising. For the case $\ov{p}=\frac12$ one could have naively
conjectured that it corresponds to total drift 0 and hence should be
recurrent. We see that this is not necessarily the case, and further,
that the question of recurrence depends also on the \emph{order} of
the $p_i$, a phenomenon which has no obvious analog with a bounded
number of cookies.

A less naive but still wrong conjecture would be ``what matters
is the \emph{average} total drift''. For example if we have 10
positive cookies followed by 10 negative cookies the ``total drift
after $n$ cookies'' fluctuates as $n$ changes between a large positive
number and 0, so maybe the average should be compared to 1. This turns
out to be wrong on two accounts. First one should not take a simple
average but a weighted average. And even then, this only explains the
numerator in the definition of $\theta$. The denominator has a
different origin, which we will explain after a necessary tour of the
Kesten-Kozlov-Spitzer approach.

The approach of Kesten, Kozlov, and Spitzer~\cite{KKS} (which may also be referred to Harris~\cite{harris1952first}) for processes on
$\Z$ is to examine the number of times the edge $(n-1,n)$ was
crossed up to a certain event --- denote it by $Z_n$ --- and study it
as a process in $n$. In the original application, random walk in
random environment, $Z_n$ behaved like a branching process with different
branching rules in different times. The approach was first applied to
excited random walk by Basdevant and Singh~\cite{basdevant2008speed},
who studied the case of finitely many
cookies, and in that case $Z_n$ turned out to be a branching process
\emph{with immigration}. In our case, however, the branching process
terminology is not as useful, and it is best to think about the
Kesten-Kozlov-Spitzer process as just some
Markov process on $\N_0=\{0,1,2,...\}$, the set of non negative integers.

We will use a variation of this approach due
to Kosygina and
Zerner~\cite{kosygina2008positively} in which there are two processes,
$Z_n^+$ and $Z_n^-$.
For a given cookie environment $\om$ define a corresponding
Markov chain $Z^+=(Z^+_n)_{n \geq 0}$ on $\N_0$ starting at $Z^+_0 = 1$, where
$Z^+_n$ counts the number of right crossings of the directed edge
$(n-1,n)$ of the ERW before hitting $-1$ for the first time.
Then, ERW on $\om$ is transient to the
right if and only if $Z^+$ does not return to zero with positive probability.
Similarly, for the Markov chain $Z^-=(Z^-_n)_{n \geq 0}$ on $\N_0$ starting
at $Z^-_0 = 1$, where $Z^-_n$ counts the number of left crossings of the
directed edge $(-n+1,-n)$ of the ERW before hitting $1$.
Then, ERW on $\om$ is transient to the left if and
only if $Z^-$ does not return to zero with positive probability.
Theorem~\ref{thm:periodic env} follows from an analysis of these two Markov chains,
together with zero-one laws for right/left transience of ERW in identically piled environments.
We shall discuss more about $Z^+$ and $Z^-$,
and their correspondence with ERW in \S\ref{sec:MC}.

Thus, the question of recurrence/transience of ERW reduces to a question
about Markov chains on $\N_0$. Our next step will be to formulate a criterion for transience of
Markov chains on $\N_0$ which is suitable for the kind of Markov
chains we will encounter.

Let $Z$ be an irreducible % aperiodic
discrete time Markov chain on $\N_0$,
and let $U$ be its step distribution. That is, for all
$n \geq 0$ the distribution of $Z_{n+1}$ conditioned on $Z_n$ is defined
as $\Pr[ Z_{n+1}=y | Z_n=x ]=\Pr[ U(x)=y ]$. Assume that the limit
$\lim_{x \to \infty} \frac{\E[U(x)]}{x}$ exists, and denote it by
\[
    \mu = \lim_{x \to \infty} \frac{\E[U(x)]}{x}.
\]
Furthermore, assume that $U$ is concentrated around its expectation. That is,
for all $x \in \N_0$ sufficiently large and for all $\eps>0$ it holds that
\[
    \Pr\left[ \left|\frac{U(x)}{x}-\mu \right| > \eps \right]
    \leq 2\exp(- C \eps' x)
\]
for some $\eps'$ that depends on $\eps$. See the statement of Theorem~\ref{thm:Fundamental}
for the precise assumptions.
We now define some quantities associated with $U$.
For each $x \in \N_0$ let
\begin{description}
\item[Drift]
    $\rho(x) = \E[U(x)-\mu x]$.
\item[Diffusion]
    $\nu(x) = \frac{\E[(U(x)-\mu x)^2]}{x}$.
\item[Ratio]
    $\theta(x) = \frac{2\rho(x)}{\nu(x)}$.
\end{description}
Note that since $Z$ is irreducible the random variable $U(x)$ is not constant.
Therefore, $\nu(x) > 0$ for all $x \in \N_0$, and so $\theta(x)$ is well defined.
We prove the following theorem.

\begin{thm}[Transience criterion for Markov chains on $\N_0$]
\label{thm:Fundamental}
Let $Z$ be an irreducible % aperiodic
discrete time Markov chain on $\N_0$ as above,
and let $U$ be its step distribution.
Assume that the limit $\mu = \lim_{x \to \infty} \frac{\E[U(x)]}{x}$ exists,
and furthermore that $U$ is concentrated in the sense that
there is a constant $C>0$ such that
for all $x \in \N_0$ sufficiently large and for all $\eps>0$ it holds that
\[
    \Pr\left[ \left|\frac{U(x)}{x}-\mu \right| > \eps \right]
    \leq 2\exp(-\frac{c \eps^2}{1+\mu+\eps}x).
\]
Suppose that $\mu \neq 1$.
\begin{itemize}
    \item If $\mu>1$, then $\Pr[Z_n>0 \mbox{ for all } n] > 0$.
    \item If $\mu<1$, then $\Pr[Z_n=0 \mbox{ for some } n] = 1$.
\end{itemize}
Suppose that $\mu=1$.
\begin{itemize}
    \item
    If $\theta(x) < 1 + \frac{1}{\ln(x)} - \alpha(x)\cdot x^{-\half}$
    for all sufficiently large $x \in \N_0$, where $\alpha(x)$ is such that $\alpha(x)\nu(x) \to+\infty$ as
    $x\to\infty$, then $\Pr[Z_n=0 \mbox{ for some } n] = 1$.
    \item
    If $\theta(x) > 1 + \frac{2}{\ln(x)} + \alpha(x)\cdot x^{-\half}$
    for all sufficiently large $x \in \N_0$, where $\alpha(x)$ is such that $\alpha(x)\nu(x) \to+\infty$ as
    $x\to\infty$, then $\Pr[Z_n>0 \mbox{ for all } n] > 0$.
    \end{itemize}
\end{thm}

\begin{rem}
    Note that Theorem~\ref{thm:Fundamental} does not cover the cases where
    $\theta(x)$ is between $1 + \frac{1}{\ln(x)}$ and $1 + \frac{2}{\ln(x)}$,
    and so it is not applicable for all Markov chains that satisfy
    $\theta = \lim_{x \to \infty} \theta(x) = 1$. Nevertheless, it will be
    enough for Theorem~\ref{thm:periodic env} since in the periodic
    case we have $|\theta(x)-\theta| \leq C\log^4(x)/{\sqrt{x}}$,
    as well as to reprove known results for bounded cookie environments,
    positive cookie environments, and branching processes with migration,
    where in all cases $|\theta(x)-\theta| = {O}(x^{-1/2})$.
\end{rem}

\begin{rem}
If $\liminf_{x\to\infty}\nu(x)>0$, then in the case $\mu=1$ in Theorem~\ref{thm:Fundamental}
it is enough to assume $\alpha(x)\to\infty$ instead of $\alpha(x)\nu(x)\to\infty$.
In particular, in the applications to Theorem~\ref{thm:Fundamental} in this paper there exists a limit $\lim_{x\to\infty}\nu(x)=\nu>0$.
\end{rem}

\begin{rem}
The denominator $(1+\mu+\eps)$ in the concentration assumption
is somewhat non-standard.
Note however, that for small values of $\eps$ this is
equivalent to the standard assumption
$\Pr\left[ \left|\frac{U(x)}{x}-\mu \right| > \eps \right] \leq 2\exp(- c \eps^2 x)$,
while for large values of $\eps$ this is essentially equivalent to
$\Pr\left[ \left|\frac{U(x)}{x}-\mu \right| > \eps \right] \leq 2\exp(- c \eps x)$.
\end{rem}

Theorem~\ref{thm:Fundamental} is quite easy to understand intuitively
even in the case that $\mu=1$. Assume $\theta(x)$ converges to some
$\theta$. Then $Z$ is a discrete version of a \emph{Bessel process} in
dimension $\theta+1$ (this connection between excited random walk and
Bessel processes has already been noted in~\cite{kosygina2009limit}).

Similar results were proved by Lamperti~\cite{Lamperti}
and Menshikov, Asymonth, and Iasnogorodski~\cite{zerodrift}
in slightly different settings. %In particular, for the case of $\mu=1$
%they assume that $\E[(U(x) - x)^2 \cdot \log^{1+\eps}(|U(x) - x|)]<C<\infty$, while in our setting
%we have $\E[(U(x) - x)^2] \ge c \cdot n$ for some constant $c>0$, and so we cannot apply their results.
Fortunately, the concentration assumptions in Theorem~\ref{thm:Fundamental}
are sufficient for our purposes, and we prove Theorem~\ref{thm:Fundamental}
following the same ideas as in~\cite{Lamperti} and~\cite{zerodrift} by using
the classic approach of Lyapunov functions. As this is standard, the
proof will be given in the appendix.

\medskip

We are now in a position to explain Theorem~\ref{thm:periodic env}. We
will show below that the $\theta$ given by~\eqref{eq:periodic theta}
is exactly the $\theta$ of Theorem~\ref{thm:Fundamental} when applied
to the process $Z_n$. In fact, the numerator of~\eqref{eq:periodic theta}
is $2\rho$ and the denominator is $\nu$. Theorem~\ref{thm:Fundamental}
can also be used to give short proofs of existing results such as in
the case of bounded environments studied in~\cite{kosygina2008positively}.
In this case the quantity $\rho$ is exactly the total drift $\sum_i(2p_i-1)$,
and $\nu$ is equal to 2. Thus, the appearance of the parameter $\nu$ (the
denominator in~\eqref{eq:periodic theta}) is another phenomenon of infinite
environments, which has no analog in bounded environments.
See \S{\ref{sec:known results}} for details.

In order to apply Theorem~\ref{thm:Fundamental}
in the case of periodic environment $\om(p)$ considered in
Theorem~\ref{thm:periodic env} we define the corresponding Markov chain
$Z^+=(Z_n^+)_{n \geq 0}$ as explained above with step distribution $U_p$. We then
do the following.
\begin{enumerate}
\item
    Formulate the measure of the corresponding step distribution $U_p$ in terms of $p$.
\item
    Prove concentration bounds for $U_p$.
\item
    Calculate the parameters $\mu$ and $\theta$ as a function of $p$.
\item
    Prove that when $\lim_{x \to \infty}\theta(x)=1$,
    the convergence of the ratio $\theta(x)$ is sufficiently fast
    (that is, faster than $\frac{1}{\ln(x)}$).
\end{enumerate}

%%%%%%%%%%%%%%%%%%%%%%%%%%%%%%%%%%%%%%%%%%%%%%%%%%%%%%%%%%%%%%%%%%%%%%%%%%%%%%%
\subsection{Structure of the paper}\label{subsec:structure}
%%%%%%%%%%%%%%%%%%%%%%%%%%%%%%%%%%%%%%%%%%%%%%%%%%%%%%%%%%%%%%%%%%%%%%%%%%%%%%%

In \S\ref{sec:MC} we define the correspondence between ERW and the
Markov chain $Z^+$ on $\N_0$, and prove some properties of the step
distribution $U_p$ defined by the environment $p$.
Theorem~\ref{thm:periodic env} is proven in \S\ref{sec:periodic}.
The proofs in this section include the calculations of $\mu$ and $\theta$,
and are rather technical.
In \S\ref{sec:known results} we reprove some existing results
on ERW for the case of positive environments~\cite{zerner2005multi} and for the case
of bounded environments~\cite{kosygina2008positively}, and a result on branching processes with migration.
We conclude the paper with some open problems in
\S\ref{sec:Open Problems}. We prove
Theorem~\ref{thm:Fundamental} in the appendix.%~\ref{sec:mc survival criterion}.

%%%%%%%%%%%%%%%%%%%%%%%%%%%%%%%%%%%%%%%%%%%%%%%%%%%%%%%%%%%%%%%%%%%%%%%%%%%%%%%
\subsection{Basic notations}\label{sec:notations}
%%%%%%%%%%%%%%%%%%%%%%%%%%%%%%%%%%%%%%%%%%%%%%%%%%%%%%%%%%%%%%%%%%%%%%%%%%%%%%%

Throughout the paper we distinguish between $\N = \{1, 2, \dots, \}$,
the set of strictly positive integers and $\N_0 = \{0,1,2,\dots\}$, the
set of non-negative integers. For a positive integer $M \in \N$ we denote
$[M] = \{1,2,\dots,M\}$.
For a vector $v=(v_1,v_2\dots) \in \R^\N$, and for $j \in \N$ denote by
$s^j(v) = (v_j,v_{j+1}\dots) \in \R^\N$, that is,
$s^j(v)$ is the shift of $v$ by $j-1$ to the left.
Similarly, for a vector $v=(v_1,\dots,v_n) \in \R^n$
we denote by $s^j(v)$ its cyclic rotation by $j-1$ to the left, i.e.,
$s^j(v) = (v_j,\dots,v_n,v_1,\dots,v_{j-1})$.

%%%%%%%%%%%%%%%%%%%%%%%%%%%%%%%%%%%%%%%%%%%%%%%%%%%%%%%%%%%%%%%%%%%%%%%%%%%%%%%
\section{Associating ERW with a Markov chain on
  \texorpdfstring{$\N_0$}{the integers}}\label{sec:MC}
%%%%%%%%%%%%%%%%%%%%%%%%%%%%%%%%%%%%%%%%%%%%%%%%%%%%%%%%%%%%%%%%%%%%%%%%%%%%%%%

In this section we are setting up the tools needed in order to prove
Theorem~\ref{thm:periodic env}. The main idea behind the proof of
Theorem~\ref{thm:periodic env} is to study a different process, which
is Markovian, unlike the original ERW.
This connection between the Markov chain $Z^+$ and right transience is well-known,
and may be found, for example, in \S 2 of~\cite{kosygina2008positively}.
%Section 2.3 of \cite{Amir2013ZeroOne}, and section 2.XXX of \cite{Amir2013mob}.
We describe the correspondence between the two processes here for the
reader's convenience.

\begin{definition}\label{def:U}
Fix an elliptic and non-degenerate cookie environment $p = (p_i)_{i \in \N}$.
For each $x \geq 0$ define a random variable $U_p(x)$ in the following way. For each $x>0$ let
\[
U_p(x)=\inf\Big\{k \in \N: \sum_{i=1}^k (1-B_i)=x\Big\}-x,
\]
where $B_i \sim B(p_i)$ are independent Bernoulli random variables.
In words, $U_p(x)$ counts the number of `successes' in a sequence
of Bernoulli trials with probabilities $p_1,p_2,\dots$ until reaching
$x$ `failures'.
%In order to relate $U_p$ to an ERW in $\omega(p)$, we think that
%upon the $i^\textrm{th}$ crossing of ERW from $n-1$ to $n$
%`success' in the $B_i$ trial means moving to $n+1$ (with probability $p_i$)
%and `failure' means going back to $n-1$.
%igor: I don't like the last sentence here. It makes no sense to me.

Finally, define a Markov chain $Z^+ = (Z^+_n)_{n \geq 0}$ on $\N_0$ where
$Z^+_0 = 1$, and $U_p$ is its step distribution. That is,
\begin{equation}
    Z^+_0=1 \qquad \mbox{and} \qquad Z^+_{n+1} \sim U_p(Z^+_n).
\end{equation}
To ensure that $Z$ be irreducible, set $U_p(0)=1$.
\end{definition}

The basic observation due to Kosygina and Zerner~\cite{kosygina2008positively} is
that if $X$ is an ERW in $\omega(p)$, then on the event $T_{-1}<\infty$
the sequence $Z_n^+$ has the same distribution as the number
of right crossings of the directed edges $(n-1,n)$ by $X$ until $T_{-1}$,
where $T_{-1} = \inf \{t \geq 0 : X_t=-1\}$ is the hitting time of $-1$ by $X$.
Moreover, on $T_{-1}=\infty$ the process $Z_n^+$ stochastically dominates the corresponding
number of left crossings. Therefore, we have $Z_n^+>0$ for all $n \in \N$ if and only if $T_{-1}=\infty$.
Since in this paper we are only interested in environments that are identically piled and elliptic
the range of the walk in such environments is a.s.\ infinite, and hence
we have $\Pr[T_{-1}=\infty]>0$ if and only if $\Pr[X_n\to\infty]>0$.
(To see why the range of the walk has to be infinite, note first that
if it is non-degenerate then its range is infinite a.s.\ by the Borel Cantelli Lemma.
Otherwise, since the environment is assumed to be elliptic and identically piled,
if $\sum_i p_i<\infty$ it is transient to the left, and if $\sum_i (1-p_i)<\infty$
it is transient to the right, again using the Borel Cantelli Lemma.)
Therefore,  $\Pr[Z_n^+>0\text{ for all $n$}]>0$ if and only if $\Pr[X_n\to\infty]>0$.
The reader is referred to \S2 of~\cite{Amir2013ZeroOne} for a complete argument using the so called arrow environments.

Analogously, we define the view of ERW ``to the left'' and associate it
with the following Markovian process $Z^-$.
Let $q$ be the cookie environment defined by $q_i = 1 - p_i$ for all
$i \in \N$. Define $U_q(x)=\inf\{k \in \N: \sum_{i=1}^k (1-B'_i)=x\}-x$,
where $B'_i \sim B(q_i)$ are independent Bernoulli random variables,
and let $Z^- = (Z^-_n)_{n \geq 0}$ be a Markov chain on $\N_0$ defined as
\begin{equation}
    Z^-_0=1 \qquad \mbox{and} \qquad Z^-_{n+1}\sim U_q(Z^-_n).
\end{equation}
Symmetrically to $Z^+$, we have $\Pr[Z_n^->0\text{ for all $n$}]>0$
if and only if $\Pr[X_n\to-\infty]>0$.

We will use the following result of Amir~et al.~from~\cite{Amir2013ZeroOne}
that asserts a zero-one law for directional transience of $X$.
\begin{thm}\label{thm:DirectionalZeroOneLaw}
    Let $p$ be an elliptic cookie environment,
    and let $X$ be an ERW in $\om(p)$.
    Then $\Pr[X_n\to+\infty],\Pr[X_n\to-\infty] \in \{0,1\}$.
\end{thm}

This implies that in order to prove that $X$ is right transient a.s.
it is enough to show that $\Pr[X_n \to +\infty] > 0$. (An analogous
equivalence holds also for left transience.)
By the discussion above we get the following corollary from
Theorem~\ref{thm:DirectionalZeroOneLaw}.
\begin{thm}\label{thm:Z char transience}
    Let $p$ be an elliptic and non-degenerate cookie environment,
    and let $X$ be an ERW in $\om(p)$.
    Then, the following holds.
    \begin{itemize}
    \item
        $\Pr[Z^+_n>0 \mbox{ for all } n] > 0$ if and only if $\Pr[X_n \to +\infty] = 1$.
    \item
        $\Pr[Z^-_n>0 \mbox{ for all } n] > 0$ if and only if $\Pr[X_n \to -\infty] = 1$.
    \item
        $\Pr[Z^+_n=0 \mbox{ for some } n] = \Pr[Z^-_n=0 \mbox{ for some } n] = 1$
        if and only if $\Pr[X_n = 0 \mbox{ i.o.}]=1$.
    \end{itemize}
\end{thm}
Therefore, in order to prove Theorem~\ref{thm:periodic env} we need to
understand when the Markov chains $Z^+$ and $Z^-$ have a positive
probability to keep above $0$ for all $n \ge 0$.

%%%%%%%%%%%%%%%%%%%%%%%%%%%%%%%%%%%%%%%%%%%%%%%%%%%%%%%%%%%%%%%%%%%%%%%%%%%%%%%
\subsection{Studying the step distribution \texorpdfstring{$U_p$}{U}}\label{subsec:U_p(x)}
%%%%%%%%%%%%%%%%%%%%%%%%%%%%%%%%%%%%%%%%%%%%%%%%%%%%%%%%%%%%%%%%%%%%%%%%%%%%%%%

In order to understand the Markov chain $Z^+$ we analyze
its step distribution $U_p$. Recall that by
Theorem~\ref{thm:Fundamental} we need to understand the relevant parameters
of $U_p$, namely $\mu$, $\rho(x)$ and $\nu(x)$. In addition, in order to
apply Theorem~\ref{thm:Fundamental} we need to show that $U_p(x)$ is concentrated
around $\mu x$ in the appropriate sense.

We start by computing the expectation parameter $\mu$ explicitly, and by
showing that $U_p(x)$ is concentrated around its expectation.
In order to do so, it will be convenient to define the random variables
\begin{equation}\label{eq:defFn}
F_n = \sum_{i=1}^n B'_i,
\end{equation} where $B'_i = 1 - B_i \sim B(q_i)$ are independent
Bernoulli random variables with $q_i = 1 - p_i$ for all $i \in \N$. Note that
by definition of $F_n$ we have
\begin{equation}\label{eq:Fn and U(x)}
\{F_n < x\}=\{U_p(x)>n-x\} \text{ and }\{F_{n-1} \geq x\}=\{U_p(x) < n-x\}.
\end{equation}

For each $n \in \N$ define
\begin{equation}\label{eq:p_n}
    \ov{p}_n=\frac{1}{n}\sum_{i=1}^n p_i.
\end{equation}

We claim that for any environment $p$ such that for some real numbers $K$ and $\ov{p}$
it holds that $|\bar{p}_n-\ov{p}|\le\frac{K}{n}$ for all $n$, we have
$\mu = \frac{\ov{p}}{1-\ov{p}}$, and $\frac{U_p(x)}{x}$ is concentrated
around $\mu$. This is proven in the following proposition.
\begin{prop}
\label{prop:concentration of U}(Concentration bound for $U_p$)
    Let $p$ be a cookie environment. Suppose it
    satisfies the assumptions as above, namely, the limit
    $\ov{p}=\lim_{n\to\infty}\overline{p}_n \in (0,1)$ exists and
    there is some $K \in \R$ such that
    $|\ov{p}_n-\ov{p}|\leq \frac{K}{n}$ for all $n \in \N$.
    Then, for all $\eps> 0$ it holds that
    \[
        \Pr\left[ \left|\frac{U_p(x)}{x}-\mu \right| > \eps \right]
        \leq 2\exp(-\frac{c \eps^2}{1+\mu+\eps}x),
    \]
    where $\mu = \frac{\ov{p}}{1-\ov{p}}$,
    and $c>0$ is some constant that depends only on $p$.
\end{prop}
Note that the bound is interesting only for $\eps > \frac{C}{\sqrt{x}}$ for some
$C>0$ sufficiently large, and so we shall assume that this is indeed the case.
\begin{proof}
    We rely on the correspondence between $U_p(x)$ and $F_n$ stated
    in~\eqref{eq:Fn and U(x)}, and use the concentration of $F_n$ in order
    to prove the proposition.
    By~\eqref{eq:Fn and U(x)} we have
    \[
        \Pr\left[ \left|\frac{U_p(x)}{x}-\mu \right| > \eps \right]
        \leq
        \Pr\left[ F_{\ceil{(1+\mu+\eps)x}} < x \right]
        +
        \Pr\left[ F_{\floor{(1+\mu-\eps)x}} \geq x \right].
    \]
    We shall bound each of the two terms using Hoeffding's inequality.

For the first term, define $n=\lceil(1+\mu+\eps)x\rceil$ and note that
\begin{align}\label{eq:EFn}
\E[F_n]&=\sum_{i=1}^nq_i
=n(1-\ov{p}_n)
\stackrel{(*)}{=}n(1-\ov{p})+O(1)\nonumber\\
&\stackrel{\clap{$\scriptstyle (**)$}}{=}\,
x\Big(1+\frac{\ov{p}}{1-\ov{p}}+\eps\Big)(1-\ov{p})+O(1)
=x+x\eps(1-\ov{p})+O(1).
\end{align}
where in $(*)$ we used the assumption on $|\ov{p}_n-\ov{p}|$, and in $(**)$
the definitions of $n$ and $\mu$. Since $F_n$ is a sum of $n$ independent
bounded summands we get from Hoeffding's inequality that
\[
\Pr[F_n<x]
\stackrel{\textrm{\eqref{eq:EFn}}}{\le}\Pr\big[F_n-\E[F_n]<-c'x\eps\big]
\le \exp(-c(\eps x)^2/n)
\le \exp(-\frac{c \eps^2}{1+\mu+\eps}x)
\]
for some constants $c,c' > 0$ that depends only on $p$.
The bound for $\Pr\big[F_{\lfloor(1+\mu-\eps)x\rfloor}\ge x\big]$ is
similar, and the proposition is proved.
\end{proof}
%%%%%%%%%%%%%%%%%%%%%%%%%%%%%%%%%%%%%%%%%%%%%%%%%%%%%%%%%%%%%%%%%%%%%%%%%%%%%%%
\subsection{The centralized second moment of \texorpdfstring{$U_p(x)$}{U}}\label{subsec:VarCalculation}
%%%%%%%%%%%%%%%%%%%%%%%%%%%%%%%%%%%%%%%%%%%%%%%%%%%%%%%%%%%%%%%%%%%%%%%%%%%%%%%

The following theorem gives an explicit formula for the second moment of $U_p(x)-x$.
Recall the definition of $\ov{p}_n$ in~\eqref{eq:p_n}
\begin{lem}\label{lem:VarCalculation}
Let $p$ be a cookie environment, and let
$U_p$ be the step distribution of the corresponding Markov chain $Z^+$.
Suppose that the limit $\lim_{n \to \infty}\ov{p}_n$ exists and equals
to $\half$. For each $n\in \N$ define
\begin{equation}\label{eq:defAn}
    A_n=\frac{1}{n}\sum_{i=1}^n p_i(1-p_i).
\end{equation}
Suppose that the limit $\lim_{n \to \infty} A_n$ also exists and is strictly positive.
Denote this limit by $A$. Assume further that there is some $K$ such that
$|\ov{p}_n-\half| \leq \frac{K}{n}$ and $|A_n-A| \leq \frac{K}{n}$
for all $n\in\N$. Then, the limit
$\lim_{x\to\infty}\frac{1}{x}\E[(U_p(x)-x)^{2}]$
exists and is equal to
\[
    \lim_{x\to\infty}\frac{1}{x}\E[(U_p(x)-x)^{2}] = 8A.
\]
Moreover, the rate of convergence is bounded by $C \cdot \log^4(x)/\sqrt{x}$,
that is, for all $x \in \N_0$ sufficiently large it holds that
\[
    \frac{1}{x} \cdot \E[(U_p(x)-x)^{2}] = 8A + O \left( \frac{\log^4(x)}{\sqrt{x}} \right),
\]
where the constant implicit in the $O()$ notation depends only on $p$.
\end{lem}
Let us sketch the argument before starting the proof proper. We write
$\E[(U_p(x)-x)^2]=\sum (2t+1)\Pr[|U_p(x)-x|>t]$. We then rewrite each
term in the sum in the language of $F_n$ using
\eqref{eq:Fn and U(x)}. But $F_n$ is just a sum of independent
variables, so we can estimate it using the Berry-Esseen theorem. This
gives a sum over $\Phi$, the Gaussian cumulative distribution
function over a (small perturbation of a) linear progression. We approximate the sum with an
integral and the integral may be calculated explicitly.

Thus the proof is quite simple in principle, but there are multiple
approximation steps each of which has to be done carefully, and the
details will fill the rest of this section.
\begin{proof}
    We start by writing the expression of $\E[(U_p(x) - x)^2]$ as a sum.
\[
    \E[(U_p(x)-x)^{2}]
        = \sum_{t=0}^{\infty}(2t+1) \cdot \Pr[ |U_p(x)-x| > t ].
\]
    Note that
    \[
        \sum_{t=0}^{\infty} \Pr[ |U_p(x)-x| > t ] = \E[|U_p(x) - x|],
    \]
    which by Proposition~\ref{prop:concentration of U} is $O(\sqrt{x})$.
    Therefore, in order to prove the lemma it is enough to show that
    \[
        \left| \frac{1}{x} \sum_{t=0}^{\infty}t \cdot \Pr[ |U_p(x)-x| > t ] - 4A \right|
        = O \left( \frac{\log^4(x)}{\sqrt{x}} \right).
    \]
    Recall the random variables $F_n$ are defined in~\eqref{eq:defFn} as
    $F_n = \sum_{i=1}^n B'_i$, where $B'_i = 1 - B_i \sim B(q_i)$ are
    independent Bernoulli random variables, and $q_i = 1-p_i$ for all
    $i \in \N$. Then, using the connection between $F$ and $U_p$
    from~\eqref{eq:Fn and U(x)} it is enough to prove that
    \begin{equation}\label{eq:H_x+T_x}
        \left| \frac{1}{x} \left( \sum_{t=0}^{\infty}t \cdot (\Pr[ F_{2x+t} < x ] +\Pr[ F_{2x-t-1} \geq x ] ) \right) - 4A \right|
        = O \left( \frac{\log^4(x)}{\sqrt{x}} \right).
    \end{equation}
    We now divide the sum into two parts, the ``head'' and the ``tail''.
    For $x,a \in \N_0$ define
    \begin{align}
        H_x(a)&= \sum_{t=0}^{\floor{a\sqrt{x}}}t \left( \Pr[ F_{2x+t} <
          x ] + \Pr[F_{2x-t-1} \geq x] \right) \label{eq:H_x(a)}\\
        T_x(a)&= \sum_{t={\floor{a\sqrt{x}}}+1}^\infty t \left( \Pr[ F_{2x+t} < x ] + \Pr[F_{2x-t-1} \geq x] \right).\label{eq:T_x(a)}
    \end{align}
    We shall take $a = a(x)$ that grows to infinity with $x$ sufficiently slow.
    The following three claims prove Lemma~\ref{lem:VarCalculation}.

    \begin{claim}\label{claim:lim H_x}
    Let $A,a>0$ and $x \in \N$ be such that $a \leq \sqrt{x}$.
    Then, for $H_x(a)$ as in~\eqref{eq:H_x(a)}
    the following holds.
    \[
        \left|
        \frac{1}{x} \cdot H_x(a) -
        \frac{1}{x} \cdot \sum_{t=0}^{\floor{a\sqrt{x}}}
        2 t \cdot \Phi \left( \frac{-t}{\sqrt{8A x}} \right)
        \right|
        \le \frac{Ca^4}{\sqrt{x}}.
    \]
    \end{claim}
    Here $\Phi$ is the cumulative distribution function of the
    normal distribution. Estimating the sum in Claim~\ref{claim:lim H_x} is a
    standard exercise in approximating sums by integrals: let us
    formulate it as a claim.
    \begin{claim}\label{claim:sum t Phi(t)}
    Let $A>0$ be a constant, and let $a = a(x)$. Then
    \[
        \lim_{x \to \infty} \frac{1}{x}
        \sum_{t=0}^{\floor{a\sqrt{x}}}
        2 t \cdot \Phi \left( \frac{-t}{\sqrt{8Ax}} \right) = 4A.
    \]
    Furthermore, the rate of convergence in at most $O(\frac{a}{\sqrt{x}})$, that is
    \[
        \frac{1}{x} \sum_{t=0}^{\floor{a\sqrt{x}}}
        2 t \cdot \Phi \left( \frac{-t}{\sqrt{8Ax}} \right) = 4A +
        O\Big(\frac{a}{\sqrt{x}} + \exp(-ca)\Big).
    \]
    \end{claim}
    Finally, for the tail we have the following estimate.
    \begin{claim}\label{claim:limT_x(a) = o(1)}
    Let $A>0$ be a constant, and let $a = a(x)$.
    Then, for all $x \in \N_0$ sufficiently large it holds that
    \[
        \frac{1}{x} \cdot T_x(a)  \leq C\exp(-ca).
    \]
    \end{claim}
\noindent
    In all three claims the constants $c$, $C$ and the constants
    implicit in the $O()$ notation depend on $p$
    but are independent of $a$ and $x$.

    The lemma follows by letting $a = K\ln(x)$ for some $K$
    sufficiently large and applying the claims.
    By Claim~\ref{claim:limT_x(a) = o(1)} we get that
    \[
        \frac{1}{x} \cdot T_x(a)  \ll \frac{1}{\sqrt{x}}.
    \]
    By combining Claims~\ref{claim:lim H_x}
    and~\ref{claim:sum t Phi(t)} we get that
    \[
        \frac{1}{x} \cdot H_x(a) = 4A +
        O\bigg(\frac{\log^4(x)}{\sqrt{x}}\bigg).
    \]
    This proves~\eqref{eq:H_x+T_x},
    which, in turn, concludes the proof of Lemma~\ref{lem:VarCalculation}.
    \end{proof}

\begin{proof}[Proof of Claim~\ref{claim:limT_x(a) = o(1)}]
    By Proposition~\ref{prop:concentration of U} the random variable
    $U_p(x)$ is concentrated, and hence
    \begin{eqnarray*}
    T_x(a)
        & = & \sum_{t=\floor{a\sqrt{x}} + 1}^\infty t \cdot \Pr[|U_p(x)-x|>t] \\
        & \leq & \sum_{i=a}^\infty \sum_{t=\floor{i\sqrt{x}} + 1}^{\floor{(i+1)\sqrt{x}}} t \cdot \exp(-\frac{ct^2}{2x+t}).
    \end{eqnarray*}
    The inner sum has at most $\sqrt{x}$ terms, each upper bounded by
    $4i\sqrt{x} \cdot \exp(-\frac{ci^2 x}{2x+i\sqrt{x}})
    \leq 4i\sqrt{x} \cdot \exp(-c' i)$.
    Therefore
    \begin{align*}
    T_x(a)
        & \leq \sum_{i=a}^\infty \sqrt{x} \cdot 4i \sqrt{x}  \cdot \exp(-c'i ) \\
        & \leq 4x\sum_{i=a}^\infty i \cdot \exp(-c'i)
        \leq Cx \exp(- c'a).\qedhere
    \end{align*}
\end{proof}

\begin{proof}[Proof of Claim~\ref{claim:lim H_x}]
    Recall the definition of $F_n$ in~\eqref{eq:defFn},
    and denote $\sigma_i^2=\E [(B'_i-q_i)^2] = p_iq_i$
    and $\rho_i=\E[|B_{i}-q_i|^3]$ for all $i \in \N$.
    By the Berry-Esseen theorem
    (\cite{berry1941accuracy,Esseen1942})
    there exists an absolute constant $C_0$ so that for all $\a \in \R$
    it holds that
    \[
        \left| \Pr\left[ \frac{ F_n-n \ov{q}_n }{ \sqrt{nA_n} } \leq \a \right] -\Phi(\a) \right|
        \leq
        C_0 \cdot \left( \sum_{i=1}^n \sigma_i ^2 \right)^{-3/2} \cdot \left( \sum_{i=1}^n \rho_i \right),
    \]
    where $A_n$ is defined in~\eqref{eq:defAn}, and $\Phi$ is the cumulative
    distribution function of the normal variable $\Norm(0,1)$.
    In our case we have $\rho_i\le1$
    for all $i \in \N$, and hence
    \[
        \left( \sum_{i=1}^n \sigma_i ^2 \right)^{-3/2} \cdot \left( \sum_{i=1}^n \rho_i \right)
        \leq  (n A_n)^{-3/2} \cdot n \leq Cn^{-1/2}.
    \]
    %where $\tilde{A}=\sup_n\{A_n^{-\frac{3}{2}}\}$.
    Therefore, for all $\a \in \R$ we have
    \begin{equation}\label{eq:BerryEsseen-A_n}
        \Pr\left[ \frac{ F_n-n \ov{q}_n }{ \sqrt{nA_n} } \leq \a
          \right] = \Phi(\a) + O(n^{-1/2}),
    \end{equation}
    where the constant implicit in the $O$ notation depends only on $p_i$'s.

    We next observe that we may replace $A_n$ with $A$ and $\bar{q}_n$
    with $\half$ in~\eqref{eq:BerryEsseen-A_n}.
    This is summarized in the following claim.
    \begin{claim}\label{claim:BerryEsseen}
    For all $n \in \N$
    and for all $\a \in \R$ such that $|\a| \leq \sqrt{\frac{n}{2A}}$ it holds that
    \[
        \Pr\left[ \frac{ F_n - \half n }{ \sqrt{nA} } \leq \a \right]
        = \Phi(\a) + O(n^{-1/2}),
    \]
    where the constant implicit in the $O$ notation depends only on $p_i$'s.
    \end{claim}

    \begin{proof}
    We show below that
    \[
        \left| \Pr \left[ \frac{F_n-n\ov{q}_n}{\sqrt{nA_n}} \leq \a \right]
            - \Pr \left[ \frac{F_n -\half  n}{\sqrt{nA}} \leq \a
              \right] \right| \le Cn^{-1/2}.
    \]
    This, together with~\eqref{eq:BerryEsseen-A_n} will imply the claim.
    \begin{align*}
     \lefteqn{\left| \Pr \left[ \frac{F_n-\ov{q}_n n}{\sqrt{nA_n}} \leq \a \right]
            - \Pr \left[ \frac{F_n - \half n}{\sqrt{nA}} \leq \a \right] \right|}\qquad\qquad&\\
     =\;&
     \left| \Pr \left[ \frac{F_n-\ov{q}_n n}{\sqrt{nA_n}} \leq \a \right]
            - \Pr \left[ \frac{F_n-\ov{q}_n n}{\sqrt{nA_n}}
                \leq \sqrt{\frac{A}{A_n}} \cdot \a + \frac{n(\half-\ov{q}_n)}{{\sqrt{n A_n}}} \right] \right| \\
%            \stackrel{(*)}{\leq} \; &
%     \left| \Pr \left[ \frac{F_n- \ov{q}_n n}{\sqrt{nA_n}} \leq \a \right]
%            - \Pr \left[ \frac{F_n- \ov{q}_n n}{\sqrt{nA_n}}
%                \leq \sqrt{\frac{A}{A_n}} \cdot \a + \frac{K}{\sqrt{n A_n}} \right] \right| \\
     \stackrel{\textrm{\eqref{eq:BerryEsseen-A_n}}}{\leq} &
     \left| \Phi(\a) - \Phi\bigg(\sqrt{\frac{A}{A_n}} \cdot \a +
                \frac{n(\frac 12-\ov{q}_n)}{\sqrt{n A_n}}\bigg)
     \right| + Cn^{-1/2}.
    \end{align*}
%    where $(*)$ is using $|\ov{q}_n - \half| \leq \frac{K}{n}$.
    Using the fact that the function $\Phi$ is $\frac{1}{\sqrt{2\pi}}$-Lipschitz together with
    the assumptions that $|A_n - A| < \frac{K}{n}$ and $|n(\frac
    12-\ov{q}_n)|\le K$ the difference is bounded by
    \begin{align*}
     \lefteqn{
       \left| \Pr \left[ \frac{F_n- \ov{q}_n n}{\sqrt{nA_n}} \leq \a \right]
            - \Pr \left[ \frac{F_n-\half n}{\sqrt{nA}} \leq \a \right] \right|}\qquad\qquad&\\
     & \leq
     \frac{1}{\sqrt{2\pi}} \left( \left|\sqrt{\frac{A}{A_n}} - 1 \right| \cdot \a
     + \frac{K}{\sqrt{n A_n}} \right)+ Cn^{-1/2} \\
     & \leq
     \frac{1}{\sqrt{2\pi}} \left( \frac{C' K}{n A} \cdot \a + \frac{K}{\sqrt{n A -
         K}} \right)+ Cn^{-1/2} \le C''n^{-1/2},
    \end{align*}
    for some constants $C,C',C'' > 0$.
    \end{proof}
    We now return to the proof of Claim~\ref{claim:lim H_x}.
    Recall the definition of $H_x(a)$ in~\eqref{eq:H_x(a)},
    \[
        H_x(a)= \sum_{t=0}^{\floor{a\sqrt{x}}}t \left( \Pr[ F_{2x+t} < x ] + \Pr[F_{2x-t-1} \geq x] \right).
    \]
    Let us rewrite $H_x(a)$ as follows
    \begin{align*}
        H_x(a)&= \sum_{t=0}^{\floor{a\sqrt{x}}}
        t \cdot
        \left( \Pr \left[ \frac{F_{2x+t}-\frac{2x+t}{2}}{\sqrt{(2x+t)A}} < \frac{-t}{2\sqrt{(2x+t)A}} \right]\right.\\
        &\qquad\qquad+ \left.\Pr \left[\frac{F_{2x-t-1} - \frac{2x-t-1}{2}}{\sqrt{(2x-t-1)A}} \geq \frac{t+1}{2\sqrt{(2x-t-1)A}} \right] \right).
    \end{align*}
    By Claim~\ref{claim:BerryEsseen} we have
    \[
        H_x(a)= \sum_{t=0}^{\floor{a\sqrt{x}}}
        t \cdot
        \left( \Phi\bigg( \frac{-t}{2\sqrt{(2x+t)A}} \bigg)
        + 1 - \Phi\bigg(  \frac{t+1}{2\sqrt{(2x-t-1)A}} \bigg) + O(x^{-1/2})
        \right),
    \]
    where the last term $O(x^{-1/2})$ is the error term from
    Claim~\ref{claim:BerryEsseen} applied twice with $n = 2x+t$ and
    with $n=2x-t-1$ for each $t \geq 0$. Therefore, by the assumption
    that $a \leq \sqrt{x}$, for $x$ sufficiently large
    the error term $O(n^{-1/2})$ is bounded by $O(x^{-1/2})$.
    Since $\Phi$ is $\frac{1}{\sqrt{2\pi}}$-Lipschitz on $\R$, it
    follows that for $t \geq 0$ we have
    \[
        \left| \Phi\left( \frac{-t}{2\sqrt{(2x+t)A}} \right)
        - \Phi\left( \frac{-t}{\sqrt{8xA}} \right) \right|
        \le
        \frac 1{\sqrt{2\pi}}\left|\frac{-t}{2\sqrt{(2x+t)A}} - \frac{-t}{2\sqrt{2xA}}\right| %\leq
%\frac{t^2}{A^{\half} x^{\frac{3}{2}}}
        \le
        \frac{Ct^2}{x^{3/2}}.
    \]
    An analogous calculation gives
    \[
        \left| \Phi\left(  \frac{t+1}{2\sqrt{(2x-t-1)A}} \right)
        - \Phi\left( \frac{t}{\sqrt{8xA}} \right) \right|
        \le \frac{Ct^2}{x^{3/2}}.
    \]
    Therefore,
    \begin{multline*}
        \left|
        H_x(a) -
        \sum_{t=0}^{\floor{a\sqrt{x}}}
        t \cdot
        \left( \Phi \left( \frac{-t}{\sqrt{8A x}} \right)
        + 1 - \Phi \left( \frac{t}{\sqrt{8A x}} \right)
        \right) \right|\\
         \leq
        C\sum_{t=0}^{\floor{a\sqrt{x}}}
        t \cdot \left(\frac{t^2}{x^{3/2}} + x^{-1/2}\right)
        \leq Ca^4 x^{1/2}.
    \end{multline*}
    The claim follows from the fact that $\Phi(\a) = 1 - \Phi(-\a)$ for all $\a \in \R$.
\end{proof}

\begin{proof}[Proof of Claim~\ref{claim:sum t Phi(t)}]
    Since the function
    $f(t)=2t \cdot \Phi\left( \frac{-t}{\sqrt{8Ax }} \right)$
    is Lipschitz on $\R$ (with a constant independent of $x$), it follows that
    \[
        \left| \sum_{t=0}^{\floor{a\sqrt{x}}} f(t) - \int_0^{a\sqrt{x}} f(t)\, dt \right| \leq Ca\sqrt{x}.
    \]
    To estimate the integral, do a linear change of variable and get
\[
    \int_0^{a\sqrt{x}} 2 t \cdot \Phi \left( \frac{-t}{\sqrt{8A x}} \right) dt
     =
        16A x \int_{-\sqrt{a/8A}}^0 u \cdot \Phi \left( u \right) du.
\]
Write
\[
\int_{-\sqrt{a/8A}}^0 = \int_{-\infty}^0 -
\int_{-\infty}^{-\sqrt{a/8A}}.
\]
The first integral can be show to be equal to $\frac 14$ with a simple
integration by parts. The second can be bounded by $C\exp(-ca)$ since
$\Phi(x)\le C\exp(-cx^2)$. This completes the proof of
Claim~\ref{claim:sum t Phi(t)} and hence also of that of
Lemma~\ref{lem:VarCalculation}.
\end{proof}

%%%%%%%%%%%%%%%%%%%%%%%%%%%%%%%%%%%%%%%%%%%%%%%%%%%%%%%%%%%%%%%%%%%%%%%%%%%%%%%
\section{Application: the periodic case}\label{sec:periodic}
%%%%%%%%%%%%%%%%%%%%%%%%%%%%%%%%%%%%%%%%%%%%%%%%%%%%%%%%%%%%%%%%%%%%%%%%%%%%%%%

In this section we prove Theorem~\ref{thm:periodic env}.
Recall that $\ov{p}$ is the average of
the $p_i$, and that $\theta$ given in \eqref{eq:periodic theta} is defined as
\begin{equation}\label{eq:periodic theta2}
        \theta(p_1,\dots,p_M)=
        \frac{\sum_{i=1}^M \delta_i (1-p_{i})}
            {4 \sum_{j=1}^M p_j(1-p_j)},
    \end{equation} where $\delta_i=\sum_{j=1}^{i} (2p_j - 1)$.
We need to prove that if $\ov{p}\ne \half$ then the transience of $Z$
depends on whether $\ov{p}$ is smaller or larger than $\half$, and if
$\ov{p}=\half$ then its transience depends on $\theta$.

In order to prove the theorem let us fix $p = (p_1,\dots,p_M) \in (0,1)^M$,
and let $\ov{p} = \frac{1}{M}\sum_{i=1}^M p_i$. Let $U_p$ be the step distribution of the Markov chain $Z^+$
defined by the environment $\om(p_1,\dots,p_M)$.
We wish to apply Theorem~\ref{thm:Fundamental}
in order to prove Theorem~\ref{thm:periodic env}.
Recall the parameters of $U_p$ considered in Theorem~\ref{thm:Fundamental}.
\begin{align*}
\mu &= \lim_{x \to \infty} \frac{\E(U_p(x))}{x} &
\rho(x) &= \E[U_p(x)-\mu x] \\
\nu(x) &= \frac{\E[(U_p(x)-\mu x)^2]}{x} &
\theta(x) &=\frac{2\rho(x)}{\nu(x)},
\end{align*}
Let $\rho$, $\nu$, and $\theta$ be the corresponding limits
whenever they exist. The following proposition supplies the ingredients
required for the proof of Theorem~\ref{thm:periodic env}.
\begin{prop}\label{prop:periodic parameters}
    Let $M \in \N$, and let $p = (p_1,\dots,p_M) \in (0,1)^M$ be a periodic cookie environment.
    Let $U_p$ be the step distribution of the corresponding Markov chain $Z^+$.
    Then
    \begin{enumerate}
    \item\label{item:periodic mu}
        We have $\mu = \frac{\ov{p}}{1-\ov{p}}$.
        In particular, $\ov{p} > \half$ if and only if $\mu > 1$,
        and $\ov{p} < \half$ if and only if $\mu < 1$.
    \item
        Suppose that $\ov{p} = \half$. Then
    \begin{enumerate}
        \item\label{item:periodic rho}
        Let
        \[
        \rho = \frac{2}{M} \sum_{i=1}^M (1-p_{i}) \cdot \sum_{j=1}^{j} (2p_i - 1).
        \]
        Then $|\rho(x) - \rho| \le \exp(- cx)$, where $c$ depends on $p$ but not on $x$.
        \item\label{item:periodic nu}
        Let
        \[
            \nu = 8 \cdot \frac{1}{M} \sum_{i=1}^M p_i(1-p_i).
        \]
        Then
        $|\nu(x) - \nu| \le C \log^4(x)/\sqrt{x}$, where $C$ depends on $p$ but not on $x$.
        \item\label{item:periodic theta}
        Let
        \[
            \theta = \frac{2 \rho}{\nu}.
        \]
        Then
        $|\theta(x) - \theta| \le C\log^4(x)/\sqrt{x}$, where $C$ depends on $p$ but not on $x$.
    \end{enumerate}
    \end{enumerate}
\end{prop}

Theorem~\ref{thm:periodic env} is now a simple corollary of
Theorem~\ref{thm:Fundamental}, Proposition~\ref{prop:periodic parameters} and Theorem~\ref{thm:Z char transience}. Here are the details:
\begin{proof}[Proof of Theorem~\ref{thm:periodic env} given
    Proposition~\ref{prop:periodic parameters}]
    Let $p = (p_1,\dots p_M) \in (0,1)^M$ be a periodic environment.
    We shall consider the Markov chains $Z^+$ and $Z^-$ defined by $\om(p)$,
    and the corresponding step distributions $U_p$ and $U_q$, where
    $q = (q_i=1-p_i)_{i \in \N}$. Recall that by
    Proposition~\ref{prop:concentration of U} the step distributions
    $U_p(x)$ and $U_q(x)$ are concentrated, as required in the conditions
    of Theorem~\ref{thm:Fundamental}.

    Suppose first that $\ov{p} > \half$ and consider the step distribution
    $U_p(x)$ that corresponds to the Markov chain $Z^+$ defined by $\om(p)$.
    Then, by the first item of Proposition~\ref{prop:periodic parameters} we have $\mu > 1$, and thus
    by Theorem~\ref{thm:Fundamental} it holds that
    $Z^+_n$ is transient. Therefore, by
    Theorem~\ref{thm:Z char transience} the ERW in $\om(p)$ is right transient
    a.s.

    Analogously, if $\ov{p} < \half$, then if we consider the Markov chain
    $Z^-$ defined by $\om(p)$ we get that $\mu < 1$, and thus
    by Theorem~\ref{thm:Fundamental} it holds that
    $Z^-_n$ is transient. Therefore, by
    Theorem~\ref{thm:Z char transience} the ERW in $\om(p)$ is left
    transient a.s.

    Suppose now that $\ov{p} = \half$, which corresponds to $\mu = 1$ for
    both $Z^+$ and $Z^-$. Suppose first that $\theta(p_1,\dots,p_M) > 1$. Then, by
    Theorem~\ref{thm:Fundamental} we have
    $Z^+_n$ is transient, and thus, by
    Theorem~\ref{thm:Z char transience} the ERW in $\om(p)$ is right
    transient a.s.

    Analogously if $\ov{p} = \half$ and $\theta(1-p_1,\dots,1-p_M) > 1$.
    Then, by Theorem~\ref{thm:Fundamental} we have
    $Z^-_n$ is transient, and thus, by
    Theorem~\ref{thm:Z char transience} the ERW in $\om(p)$ is left
    transient a.s.

    Finally, if both $\theta(p_1,\dots,p_M) \leq 1$ and
    $\theta(1-p_1,\dots,1-p_M) \leq 1$,
    then by Theorem~\ref{thm:Fundamental} we have
    $\Pr[Z^+_n=0 \mbox{ for some } n] = 1$
    and $\Pr[Z^-_n=0 \mbox{ for some } n] = 1$.
    Therefore, by Theorem~\ref{thm:Z char transience} the ERW in
    $\om(p)$ is recurrent a.s.
\end{proof}

We now turn to the proof of Proposition~\ref{prop:periodic parameters}.
Let $U_p$ be the step distribution of $Z^+$
defined by the periodic environment $\om(p_1,\dots,p_M)$.
Recall (Definition~\ref{def:U} on page~\pageref{def:U}) that $U_p(x)$ is
the number of successes in a sequence of Bernoulli trials with
periodic parameters until $x$ failures. Suppose we already counted how
many successes we had up to the first $i$ failures and we wish to
proceed to $i+1$. Because the cookies $p_i$ are periodic, we do not
need to remember our exact ``position'' in the pile of cookies, but
only its value modulo $M$. These values form a \emph{Markov chain}
with $M$ states, with $i$ being the time. Thus, we arrived at a description of $U_p$ in terms of
two sequence: the Markov chain of the values modulo $M$ (which we will
denote by $R_i$) and the number of failures at the $i^\textrm{th}$
step (which we will denote by $g_i$). Here is a more formal
description.

\begin{definition}\label{def:Rjgj}
    For a periodic cookie environment $p \in [0,1]^{\N}$, and for $j\in [M]$
    let $U^{(j)} = U_{s^j(p)}(1)$ be the number of successes in a sequence
    of Bernoulli trials with probabilities $p_j,p_{j+1},\dotsc$ until the
    first failure.%
    \footnote{Recall that $s^j(p) = (p_{j},p_{j+1},\dots,p_{j-1})$ is the
    left shift by $j-1$ of the environment $p$.}
    Define two sequences $\left(R_i \in [M]\right)_{i \geq 0}$ and
    $\left(g_i \in \N_0\right)_{i \geq 0}$ as follows.
    We start with $R_0=1$ and $g_0$ distributed as $U^{(1)}$.
    Inductively, for each $i \in \N$ define $R_i = R_{i-1}+ g_{i-1} + 1 \pmod M$,
    and define $g_i$ to be distributed as $U^{(R_i)}$. Other than the
    dependency on $R_i$, the random variable $g_i$ is independent of all
    previous $\{g_j,R_j:j < i\}$.
    Informally speaking, $R_i$ represents the location $\pmod M$ of the next
    available cookie after the $i^{\mathrm{th}}$ failure, and $g_i$ represents
    the number of successes between the $i^\textrm{th}$ and the $(i+1)^\textrm{st}$ failure.
\end{definition}
We show below that $\{R_i: i \geq 0\}$ is a Markov chain,
and that $U_p(x) = \sum_{i=0}^{x-1} g_i$.

\begin{claim}\label{claim:U = sum gi}
Let $(R_i)_{i \geq 0}$ and $(g_i)_{i \geq 0} $
be as above. Then
    \begin{enumerate}
    \item
    $\sum_{i=0}^{x-1}g_i$ is distributed according to $U_p(x)$.
    \item
    $(R_i)_{i \geq 0}$ is a Markov chain on $[M]$
    with transition matrix
    $P=(P_{j,k})_{j,k \in [M]}$ given by
    \[
        P_{j,k}=\Pr[R_i = k|R_{i-1} = j]
            =\frac{a_{j,k}(1-p_{k-1})}{1-p_1 \cdot p_2 \cdots p_M},
    \]
    where
    \[
        a_{j,k}:=\begin{cases}
          1 & j=k-1\\
          p_jp_{j+1}\cdots p_{k-2}&\textrm{otherwise.}
        \end{cases}
    \]
    \end{enumerate}
    In particular, since $p_i \in (0,1)$ for all $i \in [M]$, the Markov
    chain $(R_i)_{i \geq 0}$ is irreducible and aperiodic, and,
    therefore, has a unique stationary distribution $\pi=(\pi_1,\dots,\pi_M)$.
\end{claim}

Here and below expressions such as $p_j\cdots p_{k-2}$ should be read
``cyclically'' i.e. $p_j\cdots p_{k-2}$ if $j \leq k-2$ and
$p_j\cdots p_M p_1 \cdots p_{k-2}$ otherwise. The product always
contains between 1 and $M$ terms.

\begin{proof}
Recall that
$U_p(x)=\inf\{k \in \N: \sum_{i=1}^k (1-B_i)=x\}-x$,
where $B_i = B(p_i)$ are independent Bernoulli random variables.
Note that $g_0$ counts the number of successes until the first failure.
Hence, the $(g_0+1)^\textrm{st}$ Bernoulli trial is a failure, and $g_1$ starts
counting successes until the next failure, starting from $p_j$, where
$j = R_0 + g_0 + 1$.
The process continues until reaching $x$ failures, and $\sum_{i=0}^{x-1} g_i$
counts the number of successes until then.

For the second item, the fact that $(R_i)_{i \geq 0}$ is a Markov chain
follows from the definition of $R$, since the next step $R_i$ depends only
on $R_{i-1}$, as $g_{i-1}$ is defined by independent Bernoulli trials.

Finally we show the formula for $P_{j,k}$. We write the event ``the
first failure is when $i\equiv k-1 \pmod M$'' as a sum of the
probabilities that the first failure is at $k-1+tM$ for some
$t\in\{0,1,\dotsc\}$ (the case $t=0$ is irrelevant if $j>k-1$). We get
\begin{align*}
  P_{j,k} &= \sum_t p_j\cdot p_{j+1}\dotsb p_{k-2+tM}\cdot(1-p_{k-1+tM})\\
  &= (1-p_{k-1})\sum_{t=0}^\infty a_{j,k}\cdot(p_1\dotsb p_M)^t\\
  &= \frac{a_{j,k}\cdot(1-p_{k-1})}{1-p_1\dotsb p_M}\qedhere
\end{align*}
%where in the second and third rows we understand $p_j\dotsb p_{k-2}$
%cyclically. In particular if $j=k-1$ then this term is 1.
%This proves Claim~\ref{claim:U = sum gi}.
\end{proof}

%% In case that $j=k-1$, we have

%% \begin{eqnarray*}
%%     P_{j,k}
%%     & = & \Pr[U_j(1) \equiv j-1 \pmod M] \\
%%     & = & \sum_{t=1}^\infty \Pr[U_j(1) = Mt + j-1] \\
%%     & = & \sum_{t=1}^\infty (p_1 \cdot p_2 \cdots p_M)^t \cdot p_j \cdots p_{j-1} \cdot (1 - p_{j})  \\
%%     & = & (1 - p_{j})\cdot \sum_{t=0}^\infty (p_1 \cdot p_2 \cdots p_M)^t \\
%%     & = & \frac{(1 - p_{k-1})}{1-p_1p_2 \cdots p_M}.
%% \end{eqnarray*}

%% In case that $j\ne k-1$, we have

%% \begin{eqnarray*}
%%     P_{j,k}
%%     & = & \Pr[U_j(1) \equiv k-2 \pmod M] \\
%%     & = & \sum_{t=0}^\infty \Pr[U_j(1) = Mt + k-2] \\
%%     & = & \sum_{t=0}^\infty (p_1 \cdot p_2 \cdots p_M)^t
%%         \cdot p_j \cdot p_{j+1} \cdots p_{k-2} \cdot (1 - p_{k-1})  \\
%%     & = & p_j \cdot p_{j+1} \cdots p_{k-2} \cdot (1 - p_{k-1})
%%         \cdot \sum_{t=0}^\infty (p_1 \cdot p_2 \cdots p_M)^t \\
%%     & = & \frac{p_{j+1} \cdots p_{k-2} \cdot (1 - p_{k-1})}{1-p_1p_2 \cdots p_M},
%% \end{eqnarray*}
%% as required.
%%\end{proof}

%%%%%%%%%%%%%%%%%%%%%%%%%%%%%%%%%%%%%%%%%%%%%%%%%%%%%%%%%%%%%%%%%%%%%%%%%%%%%%%
\subsection{Calculating \texorpdfstring{$\mu$}{mu}}
%%%%%%%%%%%%%%%%%%%%%%%%%%%%%%%%%%%%%%%%%%%%%%%%%%%%%%%%%%%%%%%%%%%%%%%%%%%%%%%

We are now ready to prove the first item of Proposition~\ref{prop:periodic parameters}.
\begin{lem}\label{lem:periodic mu}{\em(Calculating $\mu$)}
    Let $(p_1,\dots,p_M) \in (0,1)^M$,
    and let $U_p$ be the corresponding step distribution in the environment $\om(p_1,\dots,p_M)$.
    Let $\mu = \smash{\displaystyle \lim_{x \to \infty}}\frac{\E[U_p(x)]}{x}$.
    Then $\mu=\frac{\ov{p}}{1-\ov{p}}$, where
    $\ov{p} = \frac{1}{M} \sum_{i=1}^M p_i$.
\end{lem}

\begin{proof}
    Recall Definition~\ref{def:Rjgj} on page~\pageref{def:Rjgj},
    where $U^{(j)}$, $R_j$ and $g_j$ are defined.
    Define an $M$-dimensional vector $E = (\E[U^{(1)}],\dots,\E[U^{(M)}]) \in \R^M$.
    We claim that
    \begin{equation}\label{eq:mu = <pi,E>}
        \mu = \ip{\pi,E} = \sum_{j=1}^M \pi_j \cdot \E[U^{(j)}].
    \end{equation}
    Indeed, by definition of $\mu$ we have
    \[
        \mu = \lim_{x \to \infty}\frac{\E[U_p(x)]}{x}
            = \lim_{x \to \infty}\frac{\sum_{i=1}^x \E[g_i]}{x}
            = \lim_{x \to \infty}\frac{\sum_{i=1}^x \E[\E[g_i|R_i]]}{x}.
    \]
    Now, since $(R_i)_{i \geq 0}$ is an irreducible and aperiodic Markov chain,
    it converges to a unique stationary distribution $\pi$, and therefore
    as $i$ grows to infinity the expectation $\E[\E[g_i|R_i]]$ converges
    to $\sum_{j=1}^M \pi_j \cdot \E[g_i | R_i = j]$,
    which is equal to $\sum_{j=1}^M \pi_j \cdot \E[U^{(j)}] = \ip{\pi, E}$.
    The following two claims provide the calculations of $\pi$ and $E$.
    \begin{claim}\label{claim:calc pi}
        The unique stationary distribution $\pi$ of the Markov chain $(R_i)_{i \geq 0}$
        is given by
        \[
            \pi_j=\frac{1-p_{j-1}}{\sum_{k=1}^{M}(1-p_k)},~ j=1,\dots,M,
        \]
        where we identify $p_0$ with $p_M$.
    \end{claim}
    \begin{claim}\label{claim:calc E(Gj)}
    For each $j \in [M]$ the expectation
    $\E[U^{(j)}]$ is equal to
    \[
        \E[U^{(j)}] = \frac{\sum_{k=j}^{j-1} p_j \cdots p_k}{1-p_1 \cdots p_M},
    \]
    where the product $p_j \cdots p_k$ is cyclic for $j > k$.
    \end{claim}
    The calculation of $\mu = \ip{\pi, E}$ is a straightforward application
    of the claims (the sum over $k$ in the formula for $\E[U^{(j)}]$
    cancels telescopically after multiplication with the terms
    $1-p_{j-1}$ in $\pi$ and summing over $j$). We omit the tedious details.
\end{proof}

%We return to the proof of Claim~\ref{claim:calc pi}.
\begin{proof}[Proof of Claim~\ref{claim:calc pi}]
We show that $(\pi P)_\ell=\pi_\ell$ for all $\ell \in [M]$,
where the matrix $P=(P_{j,k})_{j,k \in [M]}$ with
\[
    P_{j,k} = \frac{a_{j,k}(1-p_{k-1})}{1-p_1 p_2 \cdots p_M}
\]
is given by Claim~\ref{claim:U = sum gi}.
Computing $(\pi P)_\ell$ we have
\begin{align*}
    (\pi P)_\ell
    & = \sum_{j=1}^{M}\pi_j P_{j,\ell}\\
    & = \sum_{j=1}^{M} \frac{1-p_{j-1}}{\sum_{k=1}^{M}(1-p_k)}
            \cdot \frac{a_{j,\ell}(1-p_{\ell-1})}{1-p_1 p_2 \cdots p_M} \\
    & = \frac{1-p_{\ell-1}}{\sum_{k=1}^{M}(1-p_k)}
            \cdot
            \frac{\sum_{j=1}^{M} (1-p_{j-1}) \cdot a_{j,\ell}}
            {1-p_1 p_2 \cdots p_M},
\end{align*}
%where the product in the nominator of the second term should be read cyclically,
%as explained after Claim~\ref{claim:U = sum gi}.
Recalling the definition of $a_{j,\ell}$ we see that the sum in the
numerator of the second term cancels telescopically, leaving
%We get a telescopic sum in the nominator
%that is equal to
$1-p_1 p_2 \cdots p_M$.
Therefore $(\pi P)_\ell = \pi_\ell$ for all $\ell \in [M]$,
and the claim follows.
\end{proof}

%We now turn to the proof of Claim~\ref{claim:calc E(Gj)}.
\begin{proof}[Proof of Claim~\ref{claim:calc E(Gj)}]
By symmetry it is enough to calculate $\E[U^{(1)}]$.
For convenience write $a_j=p_1\cdots p_j$ for $j \in [M]$, $a_0=1$. %The calculation of $\E[U_p(x)]$
%is rather straightforward.
\begin{align*}
    \E[U^{(1)}]
    & = \sum_{\ell=0}^{\infty} \ell \cdot \Pr[U^{(1)}=\ell] \\
    & = \sum_{k=0}^{\infty} \sum_{j=0}^{M-1}
                        (kM+j) \cdot \Pr[U^{(1)}=kM+j] \\
    & = \sum_{k=0}^{\infty} \sum_{j=0}^{M-1}
                        (kM+j) \cdot (a_M)^{k} \cdot a_j \cdot (1-p_{j+1})\\
    & = M \cdot \left( \sum_{k=0}^{\infty} k \cdot (a_M)^{k}
                    \cdot \sum_{j=0}^{M-1} a_j \cdot (1-p_{j+1})
                    \right)\;+\\
    &\qquad\qquad     + \left( \sum_{k=0}^{\infty} (a_M)^{k}
                    \cdot \sum_{j=0}^{M-1} j \cdot a_j \cdot (1-p_{j+1}) \right) \\
    & = M \cdot  \left( \sum_{k=0}^{\infty}k(a_M)^{k} \right)
            \cdot \left( \sum_{j=0}^{M-1}a_j(1-p_{j+1}) \right)\;+\\
    &\qquad\qquad+ \left( \sum_{k=0}^{\infty}(a_M)^{k} \right)
                \cdot \left( \sum_{j=0}^{M-1}ja_j(1-p_{j+1}) \right)\\
    & = M \cdot \frac{a_M}{(1-a_M)^{2}} \cdot (1-a_M)
        +\frac{1}{1-a_M} \cdot \left( \sum_{j=1}^{M}a_j-Ma_M \right)\\
    & = \frac{1}{1-a_M}\sum_{j=1}^{M}a_j,
\end{align*}
as required.
\end{proof}

%%%%%%%%%%%%%%%%%%%%%%%%%%%%%%%%%%%%%%%%%%%%%%%%%%%%%%%%%%%%%%%%%%%%%%%%%%%%%%%
\subsection{Calculating \texorpdfstring{$\rho$}{rho}}\label{sec:calculating rho}
%%%%%%%%%%%%%%%%%%%%%%%%%%%%%%%%%%%%%%%%%%%%%%%%%%%%%%%%%%%%%%%%%%%%%%%%%%%%%%%

In this section we compute $\rho$ in the case $\ov{p}=\half$.
Recall that by Lemma~\ref{lem:periodic mu} this implies that $\mu=1$.

\begin{lem}\label{lem:periodic rho}
    Let $p = (p_1,\dots,p_M) \in (0,1)^M$ be a periodic
    environment with $\ov{p} = \half$.
    Let
    \[
        \rho
        = \frac{2}{M} \sum_{i=1}^M (1-p_{i}) \cdot \sum_{j=1}^{i} (2p_j - 1).
    \]
    Then $\lim_{x \to \infty} \rho(x) = \rho$.
    Furthermore for all $x \in \N_0$ we have $|\rho(x) - \rho| \le \exp(-C x)$
    for some constant $C$ that depends on $p$, but not on $x$.
\end{lem}

\begin{proof}
    We first prove that the limit $\lim_{x \to \infty} \rho(x)$ exists.
     Using the notations $U^{(j)}$, $R_j$ and $g_j$ (see Definition~\ref{def:Rjgj}) we have
    \begin{align*}
        \rho(x)
            & =   \left( \sum_{i=0}^{x-1} \E[g_i] \right) - \mu x \\
            & =   \sum_{i=0}^{x-1} \left( \E[\E[U^{(R_i)}|R_i]]  - \mu \right) \\
& \stackrel{\textrm{\clap{\eqref{eq:mu = <pi,E>}}}}{=}\;
                  \sum_{i=0}^{x-1} \sum_{j=1}^M \Pr[R_i=j] \cdot \E[U^{(j)}] - \pi_j \cdot \E[U^{(j)}] \\
            & =   \sum_{i=0}^{x-1} \sum_{j=1}^M (\Pr[R_i=j] - \pi_j) \cdot \E[U^{(j)}] \\
            & =   \sum_{j=1}^M \E[U^{(j)}] \cdot \sum_{i=1}^x (\Pr[R_i=j] - \pi_j).
    \end{align*}
    Since $(R_i)_{i \geq 0}$ is irreducible and aperiodic, it converges
    exponentially fast to the stationary distribution, that is,
    there is some $c \in \R$ and some $\a \in (0,1)$ such that
    $|\Pr[R_i=j] - \pi_j| \leq c \cdot \a^i$ for all $i \in\N$
    and for all $j \in [M]$ (see, e.g., Theorem~4.9 in~\cite{levin2009markov}).
    It now follows that $\rho(x)$ converges, and if we denote
    its limit by $\rho$, then $|\rho(x) - \rho| \le \exp(-C x)$
    for some constant $C$ that does not depend on $x$.

    Next, we turn to computing the limit $\rho$ explicitly. For
    every $j=1,\dotsc,M$, define $\rho^{(j)}$ to be the value of
    $\rho$ which corresponds to the environment $s^j(p) =(p_j,p_{j+1},\dots,p_{j-1})$.
    We are interested in $\rho^{(1)}$, and our approach will be to find
    $M$ independent linear relations between the variables
    $\rho^{(j)}$. We will also need the notations $\rho^{(j)}(x)$ and $U^{(j)}(x)$
    which are $\rho(x)$ and $U(x)$ with respect to the environment $s^j(p)$.

    {\bf Step 1.} We first extract $M-1$ relations between the
    $\rho^{(j)}$ as follows. Since $U^{(j)}$ counts successes, examine
    the very first cookie and divide according to whether is was a
    success or failure. We get the following equality
    \[
        \E[U^{(j)}(x)]=p_j \cdot ( 1+\E[U^{(j+1)}(x)] ) + (1-p_j) \cdot \E[U^{(j+1)}(x-1)].
    \]
    Subtracting $\mu x$ from both sides of the equality we get
    \begin{align*}
        \rho^{(j)}(x)
        & =  p_j \cdot ( 1+\E[U^{(j+1)}(x)] - \mu x) + (1-p_j) \cdot ( \E[U^{(j+1)}(x-1)] - \mu x) \\
        & =   p_j \cdot ( 1+\rho^{(j+1)}(x) ) + (1-p_j) \cdot (\rho^{(j+1)}(x-1) - \mu).
    \end{align*}
    Taking $x \to \infty$ we get
    \begin{equation}\label{eq:periodic rho_j formula tmp}
        \rho^{(j)} = \rho^{(j+1)} + p_j - (1-p_j) \cdot \mu.
    \end{equation}
    Recall that we assume that $\ov{p}=\half$. Therefore, by
    Lemma~\ref{lem:periodic mu} if follows that $\mu=1$.
    Hence, \eqref{eq:periodic rho_j formula tmp} gives us the constraints
    $\rho^{(j+1)} = \rho^{(j)} + 1 - 2p_j$. Summing from 1 to $j-1$ we obtain
    \begin{equation}\label{eq:periodic rho_j formula}
        \rho^{(j)} = \rho^{(1)} + \sum_{k=1}^{j-1}(1 - 2 p_k) \qquad \mbox{for all $j \in [M]$.}
    \end{equation}
    These are our first $M-1$ relations.

    {\bf Step 2.} The remaining relation will be extracted from the
    stationarity of $\pi$. If we start with $j \in [M]$ distributed according to $\pi$,
    and then wait until the first failure we get again $j$ distributed
    like $\pi$. This means that we can write
    \begin{equation}\label{eq:piandU}
    \sum_{j=1}^M\pi_j\E[U^{(j)}(x)]=x\sum_{j=1}^M\pi_j\E[U^{(j)}(1)].
    \end{equation}
    For $\rho(x)$ this gives
    \[
        \sum_{j=1}^M \pi_j \cdot \rho^{(j)}(x)
            = \sum_{j=1}^M \pi_j \cdot (\E[U^{(j)}(x)] - \mu x)
            \stackrel{\textrm{\eqref{eq:piandU}}}{=}
            x\sum_{j=1}^M \pi_j(\E[U^{(j)}(1)]-\mu)
            \stackrel{\textrm{\eqref{eq:mu = <pi,E>}}}{=} 0.
    \]
    Passing to the limit as $x$ goes to infinity we get
    \[
        \sum_{j=1}^M \pi_j \cdot \rho^{(j)} = 0.
    \]
    By Claim~\ref{claim:calc pi} we have $\pi_j =
    (1-p_{j-1})\big/\sum_{k=1}^M(1-p_k)$. Plugging this in the equation above,
    and simplifying it we get
    \[
        \sum_{j=1}^M (1-p_{j-1}) \cdot \rho^{(j)} = 0.
    \]
    Substituting $\rho_j$ with its values in~\eqref{eq:periodic rho_j formula} we get
    \[
        \sum_{j=1}^M (1-p_{j-1}) \cdot \left( \rho^{(1)} - \sum_{k=1}^{j-1}(2 p_k - 1) \right) = 0.
    \]
    Isolating the variable $\rho=\rho^{(1)}$ we finally obtain the desired formula.
    \[
        \rho \cdot {\sum_{j=1}^M (1-p_j)}
        = \sum_{j=1}^M \left[(1-p_{j-1}) \cdot \sum_{k=1}^{j-1} (2p_k - 1)\right].
    \]
    By the assumption $\ov{p}=\half$ we have $\sum_{j=1}^M (1-p_{j}) = \frac{M}{2}$ and $\sum_{j=1}^M (2p_{j}-1) = 0$. Hence,
    \[
        \rho
        = \frac{2}{M} \sum_{j=1}^M \left[(1-p_{j}) \cdot \sum_{k=1}^{j} (2p_k - 1)\right].
    \]
    This completes the proof of Lemma~\ref{lem:periodic rho}.
\end{proof}

We finally prove Proposition~\ref{prop:periodic parameters}.

\begin{proof}[Proof of Proposition~\ref{prop:periodic parameters}]
Lemma~\ref{lem:periodic mu} proves Item~(\ref{item:periodic mu}) of the proposition.
Item~(\ref{item:periodic rho}) is proven in Lemma~\ref{lem:periodic rho}.
In order to prove Item~(\ref{item:periodic nu})
note first that
$|\ov{p}_n-\frac{1}{2}|\le \frac{M}{n}$ and
$A_n=\frac{1}{n}\sum_{i=1}^n p_i(1-p_i) \to \frac{1}{M}\sum_{i=1}^M p_i(1-p_i)=:A$ as $n\to\infty$.
Moreover, $A>0$ as $p_i\in(0,1)$ and $|A_n-A|<\frac{M/4}{n}$.
Therefore, Item~(\ref{item:periodic nu}) is a direct application of Lemma~\ref{lem:VarCalculation}.
Item~(\ref{item:periodic theta}) now follows from Items~(\ref{item:periodic rho})
and~(\ref{item:periodic nu}) using the fact that $A_n \to \frac{1}{M}\sum_{i=1}^M p_i(1-p_i)=\nu/8>0$
and the triangle inequality.
\end{proof}

%%%%%%%%%%%%%%%%%%%%%%%%%%%%%%%%%%%%%%%%%%%%%%%%%%%%%%%%%%%%%%%%%%%%%%%%%%%%%%%
\subsection{A concrete example of a periodic environment}
%%%%%%%%%%%%%%%%%%%%%%%%%%%%%%%%%%%%%%%%%%%%%%%%%%%%%%%%%%%%%%%%%%%%%%%%%%%%%%%

In this section we provide a concrete example of a periodic environment.
Let $M \in \N$ be an even integer, and let $p \in (0,1)$ be a parameter.
Define a periodic environment $\om(p,M)$ with first $M/2$
cookies being $p$, and the last $M/2$ cookies being $1-p$.
The average of the cookies in a period is equal to $\half$ and hence $\mu=1$.
By Lemma~\ref{lem:periodic rho} we have
\[
    \rho
    = \frac{2}{M} \sum_{i=1}^M (1-p_{i}) \cdot \sum_{j=1}^{i} (2p_j - 1).
\]
A tedious calculation gives
\[
    \rho = (2p-1) \frac{M}{4} - \frac{(2p-1)^2}{2}.
\]
%----------------------------------------------
% Letting $q=1-p$ we have
%\begin{eqnarray*}
%    \rho
%    & = & \frac{2}{M} \sum_{j=2}^M (1-p_{j-1}) \cdot \sum_{k=1}^{j-1} (2p_k - 1) \\
%    & = & \frac{2}{M} \left( \sum_{j=2}^{M/2+1} (1-p) \cdot (j-1) \cdot (2p-1)
%                    + \sum_{j=M/2+2}^M (1-q) \left(M/2 \cdot (2p-1) + (j-\frac{M}{2}-1)(2q-1) \right) \right) \\
%    & = & \frac{2}{M} \left( (1-p) (2p-1) \frac{M/2 \cdot (M/2+1)}{2}
%                    + p \cdot (2p-1) \cdot (M/2-1)\cdot M/2 + p(1-2p)\frac{(M/2-1) \cdot M/2}{2} \right) \\
%    & = & \frac{2}{M} \left( (1-p) (2p-1) \frac{M/2 \cdot (M/2+1)}{2}
%                    + p(2p-1)\frac{(M/2-1) \cdot M/2}{2} \right) \\
%    & = & (1-p) (2p-1) \frac{M/2+1}{2}
%                    + p(2p-1)\frac{(M/2-1)}{2} \\
%    & = & (2p-1) \frac{M}{4} -(2p-1)^2/2 \\
%\end{eqnarray*}
%----------------------------------
By Lemma~\ref{lem:VarCalculation} we have
\[
    \nu = \frac{8}{M}\sum_{i=1}^M p_i (1-p_i) = 8p(1-p),
\]
and hence,
\[
    \theta = \frac{2\rho}{\nu} = \frac{(\frac{M}{2} - (2p-1)) \cdot (2p-1)}{8 p (1-p)}.
\]

Therefore, by Theorem~\ref{thm:periodic env} we have the following corollary.

\begin{corollary}
Let $p \in (\half,1)$, and let $M$ be an even positive integer.
Define a periodic environment $\om(p,M)$ with first $M/2$
cookies having probabilities $p$, and the last $M/2$ cookies having probabilities $1-p$.
Then, ERW in $\om(p,M)$ is right transient if and only if
$M > \frac{8p-8p^2+2}{2p-1}$, and is recurrent otherwise.

In particular for $M=2$ ERW in the periodic environment $\om(p,1-p)$
is a.s. recurrent for all $p \in (0,1)$.
\end{corollary}

%%%%%%%%%%%%%%%%%%%%%%%%%%%%%%%%%%%%%%%%%%%%%%%%%%%%%%%%%%%%%%%%%%%%%%%%%%%%%%%
\section{More applications: reproving known results}\label{sec:known results}
%%%%%%%%%%%%%%%%%%%%%%%%%%%%%%%%%%%%%%%%%%%%%%%%%%%%%%%%%%%%%%%%%%%%%%%%%%%%%%%

In this section we show how to use Theorem~\ref{thm:Fundamental} in order
to reprove transience criterion for several known cases of ERW in identically
piled environments. We shall assume that the discussed environments
$p$ are always non-degenerate. %
%\footnote{\label{fn:delta=infty}
In the case that the environment is degenerate, then we must have
that either $p_i\to 0$ or $p_i\to 1$, which clearly imply transience.
For example, assume that $p_i \to 1$. Then, transience can be proven
by coupling the Kesten-Kozlov-Spitzer process $Z^+$ with the
corresponding process in bounded environment $p'$ defined by $p'_i:=p_i$
for all $i\leq M$ and $p'_i = \half$ for all $i>M$ for $M$ sufficiently
large to make sure that $p_i>\half$ for all $i>M$ and also
$\sum_{i=1}^M(2p_i-1)>1$ (we will explain this coupling in detail
below, in the proof of Claim \ref{claim:positive env rho(x) to delta}).
Transience in such environment follows from Theorem~\ref{thm:bddEnv}.%}

%%%%%%%%%%%%%%%%%%%%%%%%%%%%%%%%%%%%%%%%%%%%%%%%%%%%%%%%%%%%%%%%%%%%%%%%%%%%%%%
\subsection{ERW in bounded environments}\label{subsec:bounded env}
%%%%%%%%%%%%%%%%%%%%%%%%%%%%%%%%%%%%%%%%%%%%%%%%%%%%%%%%%%%%%%%%%%%%%%%%%%%%%%%

In this section we reprove the following theorem of Kosygina and Zerner
from~\cite{kosygina2008positively}
(the original proof applies in the more general setting of random environments).
\begin{thm}[Kosygina-Zerner~\cite{kosygina2008positively}]\label{thm:bddEnv}
    Let $p = (p_i)_{i \in \N}$ be an elliptic bounded cookie environment.
    That is, $p_i \in (0,1)$ for all $i \in\N$, and there is some $M \in \N$
    such that $p_i=\half$ for all $i>M$.
    Let
    \[
        \delta=\sum_{i=1}^M(2p_i-1).
    \]
    Let $X = (X_n)_{n \geq 0}$ be a ERW in $\om(p)$. Then
    \begin{enumerate}
    \item
        If $\delta>1$ then $X_n \to +\infty$ a.s.
    \item
        If $\delta<-1$ then $X_n \to -\infty$ a.s.
    \item
        If $-1 \leq \delta \leq 1$, then $X_n=0$ i.o.\ a.s.
    \end{enumerate}
\end{thm}

\begin{proof}
    Consider the step distribution $U_p$ of the Markov chain $Z^+$
    defined by the environment $p$. We start the proof by computing the
    expectation $\E[U_p(x)]$ for all $x > M$. Let $L$ be
    the number of failures in the first $M$ Bernoulli trials. Then
    \begin{equation}\label{eq:fin env E[U(x)] using L}
        \E[U_p(x)] = M - \E[L] + \E\left[ \E\Big[NB\Big(x-L,\half\Big)\Big|L\Big] \right],
    \end{equation}
    where $NB(x-L,\half)$ is the negative binomial distribution.
    Indeed, the last term is $\E \left[ \E[NB(x-L,\half)|L] \right]$
    due to the assumption that there are at most $M$
    biased cookies. Thus, after $M$ Bernoulli trials the rest are just
    are $B(p_i = \half)$ for all $i > M$, and we count the number of successes
    in unbiased Bernoulli trials until reaching additional $x-L$ failures.

    By definition $\E[L]$ is equal to
    \[
        \E[L] = \sum_{i=1}^M (1-p_i) = M - \sum_{i=1}^M p_i = \frac{M}{2} - \frac{\delta}{2}.
    \]
    The last term in~\eqref{eq:fin env E[U(x)] using L} is equal to
    \[
        \E\left[ \E\Big[NB\Big(x-L,\half\Big)\Big|L\Big] \right] = \E[x - L] = x - \E[L].
    \]
    Therefore, for $x > M$ we have
    \begin{equation}\label{eq:fin env E[U(x)]}
        \E[U_p(x)] = x + M - 2 \E[L] = x + \delta.
    \end{equation}
    That is, in the setting of Theorem~\ref{thm:Fundamental} the parameters $\mu$ and $\rho(x)$
    for $U_p(x)$ are
    \begin{equation}\label{eq:fin_env mu}
        \mu = 1.
    \end{equation}
    \begin{equation}\label{eq:fin_env rho}
        \rho(x) = \delta \qquad \mbox{ for all } x > M.
    \end{equation}
    In order to compute $\nu(x)$ we assume again that $x > M$ and compute
    $\E[(U_p(x) - x)^2]$.
    Note that for $x > M$ we can write $U_p(x) = U_p(M) + NB(x-M,\half)$,
    where the two summands are independent.
    Therefore, if we let $c = \E [ (U_p(M) - M)^2 ] < \infty$, then
    \begin{align*}
        \E[(U_p(x) - x)^2]
            & = \E \left[ \left( (U_p(M) - M) + \Big(NB\Big(x-M,\half\Big) - (x-M)\Big) \right)^2 \right] \\
            & = \E [ (U_p(M) - M)^2 ] + \E \left[ \Big(NB\Big(x-M,\half\Big) - (x-M)\Big)^2 \right] \\
            & = c + 2(x - M),
    \end{align*}
    where the second equality is by independence of $U_p(M)$ and $NB(x-M,\half)$.
    This gives us that
    \begin{equation}\label{eq:fin_env nu}
        \nu(x) = \frac{\E[(U_p(x) - x)^2]}{x} = 2 + \frac{c - 2M}{x} = 2 + O\Big(\frac{1}{x}\Big).
    \end{equation}
    Using~\eqref{eq:fin_env rho} and~\eqref{eq:fin_env nu} we get that
    for all $x > M$ it holds that
    \begin{equation}\label{eq:fin_env theta}
        \theta(x) = \frac{2\rho(x)}{\nu(x)} = \delta + O\Big(\frac{1}{x}\Big).
    \end{equation}

    Next we apply Theorem~\ref{thm:Fundamental} on $Z^+$. Recall that by
    Proposition~\ref{prop:concentration of U} the step distributions $U_p(x)$
    is concentrated, as required in the conditions of
    Theorem~\ref{thm:Fundamental}. By applying Theorem~\ref{thm:Fundamental}
    we conclude that the Markov chain $Z^+$ that corresponds to ERW in
    $\om(p)$ is transient if and only
    if $\delta > 1$. Therefore, by Theorem~\ref{thm:Z char transience} ERW in
    $\om(p)$ is right transient a.s. if and only if $\delta > 1$.

    Analogously the Markov chain $Z^-$ that corresponds to ERW in $\om(p)$ is
    transient if and only if $\delta < -1$. and hence ERW in $\om(p)$ is left
    transient a.s. if and only if $\delta < -1$.

    Lastly, if $\delta \in [-1,1]$, then both $Z^+$ and $Z^-$ are a.s. recurrent,
    and thus ERW in $\om(p)$ visits the origin i.o. a.s.
\end{proof}

%%%%%%%%%%%%%%%%%%%%%%%%%%%%%%%%%%%%%%%%%%%%%%%%%%%%%%%%%%%%%%%%%%%%%%%%%%%%%%%
\subsection{ERW in positive environments}\label{subsec:positive env}
%%%%%%%%%%%%%%%%%%%%%%%%%%%%%%%%%%%%%%%%%%%%%%%%%%%%%%%%%%%%%%%%%%%%%%%%%%%%%%%

In this section we assume that our cookie environments $p$ are positive,
that is $p_i \geq \half$ for all $i \in\N$,
and reprove the following theorem of Zerner~\cite{zerner2005multi}
(the original proof applies in the more general setting of random environments).

\begin{thm}[Zerner \cite{zerner2005multi}]\label{thm:positive env}
    Let $p = (p_i)_{i \in \N}$ be an elliptic and positive cookie environment,
    and let
    \[
        \delta=\sum_{i=1}^{\infty}(2p_i-1).
    \]
    Let $X = (X_n)_{n \geq 0}$ be a ERW in $\om(p)$.
    Then
    \begin{enumerate}
    \item
        If $\delta>1$ then $X_n \to +\infty$ a.s.
    \item
        Otherwise $X_n=0$ i.o. a.s.
    \end{enumerate}
\end{thm}

\begin{proof}
    Note first that $\delta=\infty$, then the walk is right transient. This
    can be shown by coupling the Kesten-Kozlov-Spitzer process $Z^+$
    with a corresponding process in bounded environment as explained
    in Claim \ref{claim:positive env rho(x) to delta}. %Footnote~\ref{fn:delta=infty} in
%    beginning of \S\ref{sec:known results}. \footnote{
Actually this coupling can be done for all $\delta>1$ and it is left
to prove the recurrence part. However, we prefer here to show how to
deduce it from Theorem~\ref{thm:Z char transience}. %}
    Suppose now that $\delta < \infty$. We prove the theorem by considering
    the step distribution $U_p(x)$ of the corresponding Markov chain $Z^+$,
    and computing the corresponding parameters $\mu$ and $\theta$.

    \begin{lem}\label{lem:positive env mu rho}
        Let $p$ be a positive and elliptic cookie environment.
        Suppose that $\delta = \sum_{i=1}^\infty (2p_i-1) < \infty$.
        Let $U_p$ be the step distribution of the corresponding Markov
        chain $Z^+$. Then
        $\lim_{x \to \infty}\rho(x) = \lim_{x \to \infty}(\E[U_p(x)]-x) = \delta$.
        Furthermore, $\rho(x) \leq \delta$ for all $x \geq 0$.
    \end{lem}

    \begin{lem}\label{lem:positive env nu}
        Let $p$ be a positive and elliptic cookie environment.
        Suppose that $\delta < \infty$.
        Let $U_p$ be the step distribution of the corresponding Markov
        chain $Z^+$. Then $\nu(x) = \frac{1}{x}\E[(U_p(x) - x)^2] \to 2$.
        Furthermore, for all $x \in \N_0$ sufficiently large we have
        $|\nu(x) -2 | \leq C \log^4(x)/\sqrt{x}$ for some constant $C \in \R$
        that depends only on $p$.
    \end{lem}

    The following corollary is immediate from Lemmas~\ref{lem:positive env mu rho}
    and~\ref{lem:positive env nu}.
    \begin{corollary}\label{cor:positive env}
        Let $p$ be a positive and elliptic cookie environment.
        Suppose that $\delta< \infty$.
        Let $U_p$ be the step distribution of the corresponding Markov
        chain $Z^+$. Then
        \begin{enumerate}
        \item
            $\mu = \lim_{x \to \infty}\frac{\E[U_p(x)]}{x} = 1$.
        \item
            $\lim_{x \to \infty}\theta(x) = \lim_{x \to \infty} \frac{2\rho(x)}{\nu(x)} = \delta$.
        \item
            For all $x \in \N_0$ sufficiently large we have
            $\theta(x) \leq \delta + C \cdot \log^4(x)/\sqrt{x}$ for some
            constant $C \in \R$ that depends only on $p$.
        \end{enumerate}
    \end{corollary}

    Theorem~\ref{thm:positive env} follows by applying Theorem~\ref{thm:Fundamental}
    with the parameters given in Corollary~\ref{cor:positive env}, together with
    Theorem~\ref{thm:Z char transience}.
    Consider the step distribution $U_p$ of the Markov chain $Z^+$
    defined by $p$, and recall that by Proposition~\ref{prop:concentration of U}
    we have concentration of $\frac{U_p(x)}{x}$ around $\mu$ as required in the
    conditions of Theorem~\ref{thm:Fundamental}.
    By the first item of Corollary~\ref{cor:positive env} we have that $\mu = 1$.

    If $\delta > 1$, then by the second item of
    Corollary~\ref{cor:positive env} we have $\lim_{x \to \infty}\theta(x) = \delta > 1$,
    and thus, by Theorem~\ref{thm:Fundamental} $Z^+$ is transient a.s.
    Therefore, by Theorem~\ref{thm:Z char transience} ERW in $\om(p)$ is right transient a.s.

    Suppose now that $\delta \leq 1$. Then, by the second and the third items
    of Corollary~\ref{cor:positive env} we have $\theta(x) \leq 1 + O(\log^4(x)/\sqrt{x})$
    for all $x \in \N_0$ sufficiently large,
    and thus by Theorem~\ref{thm:Fundamental} $Z^+$ is recurrent a.s.
    Therefore, by Theorem~\ref{thm:Z char transience} ERW in $\om(p)$ is not right transient a.s.
    In order to see that ERW in $\om(p)$ cannot be left transient either we can
    couple the Markov chain $Z^-$ with the one defined by a simple random walk on $\Z$.
    Therefore, if $\delta \leq 1$, then ERW on $\om(p)$ returns to the origin i.o. a.s.
\end{proof}

We now turn to prove Lemmas~\ref{lem:positive env mu rho}
and~\ref{lem:positive env nu}.

%%%%%%%%%%%%%%%%%%%%%%%%%%%%%%%%%%%%%%%%%%%%%%%%%%%%%%%%%%%%%%%%%%%%%%%%%%%%%%%
\subsubsection{Proof of Lemma~\ref{lem:positive env mu rho}}
%%%%%%%%%%%%%%%%%%%%%%%%%%%%%%%%%%%%%%%%%%%%%%%%%%%%%%%%%%%%%%%%%%%%%%%%%%%%%%%

We start with the first part of the lemma.
\begin{claim}\label{claim:positive env rho(x) to delta}
    Let $p$ be a positive and elliptic cookie environment,
    and let $\delta < \infty$.
    Then $\lim_{x \to \infty}\rho(x) = \lim_{x \to \infty}(\E[U_p(x)]-x) = \delta$.
\end{claim}

\begin{proof}
    Fix $\eps > 0$ sufficiently small. We claim that there is some $M \in \N$ large enough
    so that $|\E[U_p(x) -x] - \delta| \leq \eps$ for all $x \geq M$.

    Let $M$ be sufficiently large so that
    $\sum_{i=M}^\infty (2p_i-1) < \frac{\eps}{2}$.
    For $x > M$ define a bounded environment by $p'$ by
    letting $p'_i=p_i$ for $i < x$ and $p'_i=\half$ for all $i \geq x$. That
    is, $p'$ is obtained from $p$ by ``forgetting'' all its cookies above
    level $M$. Then
    $\sum_{i=1}^\infty |p_i-p'_i| = \half \sum_{i=x}^\infty (2p_i-1) < \frac{\eps}{4}$.
    Since $p'$ is a bounded environment, by~\eqref{eq:fin env E[U(x)]} we have
    \[
        |\E[U_{p'}(x)] - x - \delta| \leq
        |\E[U_{p'}(x)] - x - \sum_{i=1}^{x}(2p_i-1)| + \frac{\eps}{2}=\frac{\eps}{2},
    \]
    and so, it is left to prove that
    \begin{equation}\label{eq:E[Up] diff}
        | \E[U_p(x)] - \E[U_{p'}(x)] | < \frac{\eps}{2}.
    \end{equation}
    We prove~\eqref{eq:E[Up] diff} by coupling the two processes in
    the natural way. For each $i \in\N$ let $Y_i \sim U[0,1]$ be
    i.i.d.\ uniform random variables. Define
    $U_p(x)=\inf\{k \in \N : \sum_{i=1}^k {\bf 1}_{[Y_i > p_i]}=x\}-x$,
    and analogously let
    $U_{p'}(x)=\inf\{k \in \N: \sum_{i=1}^k {\bf 1}_{[Y_i > p'_i]}=x\}-x$.
    Clearly both $U_p(x)$ and $U_{p'}(x)$ have the correct distribution.
    In addition we have $U_p(x) \geq U_{p'}(x)$.
    Let $T = U_{p'}(x)+x$ be the time when $U_{p'}(x)$ reaches $x$ failures,
    and let $K = x - \sum_{i=1}^T {\bf 1}_{[Y_i > p_i]}$, be the number of
    failures of $U_p(x)$ after time $T$.
    Then
    \[
        |U_p(x)-U_{p'}(x)|\sim U_{s^{T+1}(p)}(K),
    \]
    where $s^{T+1}(p) = (p_{T+1},p_{T+2},\dots)$ is the right shift of the
    cookie environment $p$. Taking the expectation on both sides, we get
    \[
        \E[ |U_p(x)-U_{p'}(x)| ] = \E[ U_{s^{T+1}(p)}(K) ] \leq \a \cdot \E[K],
    \]
    where $\a = \sup_{k \geq M} \{\E[U_{s^k(p)}(1)] \}$.
    We show below that $\E[K] = \frac{\eps}{4}$ and $\a \leq 2$,
    which is clearly enough in order to prove~\eqref{eq:E[Up] diff},
    since $T \geq x > M$,

    In order to see that $\E[K] = \frac{\eps}{4}$
    note that $K \leq \sum_{i=1}^\infty {\bf 1}_{[p'_i < Y_i \leq p_i]}$.
    Therefore, taking the expectation we get
    \[
        \E[K] \leq \sum_{i=1}^\infty |p_i - p'_i| < \frac{\eps}{4}.
    \]
    In order to prove that $\a < 2$ note that in every environment $p$ we have
    \begin{equation}\label{eq:ExpecOfMultyGeomRV}
        \E[ U_p(1) ] = \sum_{n=1}^\infty \Pr[U_p(1) \geq n] = \sum_{n=1}^\infty \prod_{i=1}^n p_i.
    \end{equation}
    In particular, if for some $\gamma  < 1$
    it holds that $p_i < \gamma$ for all $i \geq k$ ,
    then
    \begin{equation}\label{eq:E[U(1)] for elliptic p}
        \E[U_{s^k(p)}(1)]
        \leq \sum_{n=1}^\infty \gamma^n = \frac{\gamma}{1-\gamma}.
     \end{equation}
    Recall that $M$ is sufficiently large
    so that $\sum_{i=M}^\infty (2p_i-1)<\frac{\eps}{2}$, and in particular
    $p_i < \half + \frac{\eps}{4}$ for all $i \geq M$.
    Therefore, it follows that
    $\a < \frac{\half + \frac{\eps}{4}}{\half - \frac{\eps}{4}} < 2$
    for all $\eps < 2/3$.
%    In particular, if for some $\gamma \in (0,\half)$
%    it holds that $p_i \in [\gamma,1-\gamma]$ for all $i \geq k$ ,
%    then
%    \begin{equation}\label{eq:E[U(1)] for elliptic p}
%        \frac{\gamma}{1-\gamma} = \sum_{n=1}^\infty \gamma^n
%        \leq \E[U_{s^k(p)}(1)]
%        \leq \sum_{n=1}^\infty (1-\gamma)^n = \frac{1-\gamma}{\gamma}.
%     \end{equation}
%    Recall that $M$ is sufficiently large such that
%    $\sum_{i=M}^\infty (2p_i-1)<\frac{\eps}{2}$, and in particular
%    $p_i \in (\half-\frac{\eps}{4}, \half + \frac{\eps}{4})$ for all $i \geq M$.
%    Therefore, since $T \geq x > M$, this implies that
%    $\a < \frac{\half + \frac{\eps}{4}}{\half - \frac{\eps}{4}} < 2$
%    for all $\eps > 0$ sufficiently small.
    This completes the proof of Claim~\ref{claim:positive env rho(x) to delta}
\end{proof}

Next, we prove the ``furthermore'' part of Lemma~\ref{lem:positive env mu rho}.

\begin{claim}\label{claim:positive env rho(x) < delta}
    Let $p$ be a positive and elliptic cookie environment,
    and let $\delta < \infty$.
    Then $\rho(x) \le \delta$ for all $x \geq 0$.
\end{claim}

\begin{proof}
By Claim \ref{claim:positive env rho(x) to delta}, $\lim_{x\to\infty}\rho(x)=\delta$ and so the claim will follow once we show that $\rho(x)$ is monotonically increasing in $x$. Note that for $p=(p_1,p_2,...)$, we have that $\E[U_p(1)]$ is monotonically increasing in each $p_i$. Indeed, this can be seen either from the explicit formula \eqref{eq:ExpecOfMultyGeomRV}, or using the natural coupling specified in the proof of Claim \ref{claim:positive env rho(x) to delta}. (Actually, for every $x$ it holds that $\E[U_p(x)]$ is monotonically increasing in each $p_i$, but we do not use that.) By comparing to the constant $1/2$ environment we observe that for a positive environment $p$ it holds that $\E[U_p(1)-1]\ge 0$. Therefore,
$\rho(x+1)=
\E[U_p (x+1)-(x+1)] =
\E[U_p (x) - x + U_{p'} (1) - 1)] =
\E[U_p (x) - x] + E[U_{p'} (1) +1)] \ge
\E[U_p (x) - x]=
\rho(x)$,
where $p'$ is some random (but a.s.\ finite) shift of $p$ and hence also positive, and the inequality follows from the last observation.
\end{proof}

%%%%%%%%%%%%%%%%%%%%%%%%%%%%%%%%%%%%%%%%%%%%%%%%%%%%%%%%%%%%%%%%%%%%%%%%%%%%%%%
\subsubsection{Proof of Lemma~\ref{lem:positive env nu}}
%%%%%%%%%%%%%%%%%%%%%%%%%%%%%%%%%%%%%%%%%%%%%%%%%%%%%%%%%%%%%%%%%%%%%%%%%%%%%%%

    The lemma is an immediate consequence of
    Lemma~\ref{lem:VarCalculation}. Indeed, since
    $\sum_{i=1}^\infty (2p_i-1) < \infty$ it follows that
    $|\ov{p}_n - \half| = |\frac{1}{n}\sum_{i=1}^n p_i - \half|
    = \frac{1}{2n} \sum_{i=1}^n (2p_i - 1) \leq \frac{\delta}{2n}$.
    Therefore, by Lemma~\ref{lem:VarCalculation} the limit
    of $\frac{\E[(U_{ p }(x)-x)^2]}{x}$ as $x$ tends to infinity exists, and is equal to
    \[
        \lim_{x \to \infty} \frac{1}{x}\E[(U_{ p }(x)-x)^2] = 8A,
    \]
    where
    \[
        A = \lim_{n \to \infty} A_n
          = \lim_{n \to \infty} \frac{1}{n}\sum_{i=1}^n p_i(1-p_i) = \frac{1}{4}.
    \]
    By the ``moreover'' part of Lemma~\ref{lem:VarCalculation}
    it follows that the rate of convergence is bounded by $C \cdot \log^4(x)/\sqrt{x}$,
	that is, for all $x \in \N_0$ sufficiently large it holds that
	\[
    	\left| \frac{1}{x} \cdot \E[(U_p(x)-x)^{2}] - 2 \right| =
    	O \left( \frac{\log^4(x)}{\sqrt{x}} \right),
	\]
	where the constant implicit in the $O()$ notation depends only on $p$.
	This completes the proof of Lemma~\ref{lem:positive env nu}. \qed

%%%%%%%%%%%%%%%%%%%%%%%%%%%%%%%%%%%%%%%%%%%%%%%%%%%%%%%%%%%%%%%%%%%%%%%%%%%%%%%
\subsection{Branching process with migration}\label{subsec:BPwM}
%%%%%%%%%%%%%%%%%%%%%%%%%%%%%%%%%%%%%%%%%%%%%%%%%%%%%%%%%%%%%%%%%%%%%%%%%%%%%%%

As a corollary from Theorem~\ref{thm:Fundamental} we obtain the following
result on branching processes with migration.
In order to define branching process with migration let $\xi$ and $\eta$ be
two random variables, where the support of $\xi$ is $\N_0$ and $\eta \in \Z$.
Suppose that both $\xi$ and $\eta$ have an exponential tail. That is,
there is some $\a>0$ and $t_0$ such that $\Pr[\xi > t] \le \exp(-\a t)$
and $\Pr[|\eta| > t] \le \exp(-\a t)$ for all $t>t_0$.

Let $\mu = \E[\xi]$, $\rho = \E[\eta]$, $\nu = \Var[\xi]$,
and let $\theta = \frac{2\rho}{\nu}$. Note that by the assumption on $\xi$
and $\eta$ all these quantities are finite. For $i,n,m \in\N$
let $\xi_i^{(n)}$ and $\eta_{(m)}$ be independent random variables so that
$\xi_i^{(n)} \sim \xi$ and $\eta_{(m)}\sim\eta$.

A branching process with migration is a random sequence $Z=(Z_n)_{n \geq 0}$
defined by setting $Z_0 = 1$, and for each $n \geq 0$ the random variable
$Z_{n+1}$ conditioned on $Z_n$ is distributed as
\[
    Z_{n+1} =
            \begin{cases}
                \max\left\{\sum_{i=1}^{Z_n}\xi_i^{(n)} + \eta_{(n+1)},0\right\} & \mbox{ if } Z_n > 0 \\
                0 & \mbox{ if } Z_n = 0.
            \end{cases}
\]
The random variable $\xi$ is the \emph{offspring distribution},
and $\eta$ is the \emph{migration distribution}.

We say that the process $Z$ \emph{survives} if $Z_n > 0$ for all $n$ (equivalently, the Markov chain $Z$ is transient). Otherwise we
say that $Z$ dies out (equivalently, the Markov chain $Z$ is recurrent). The following theorem gives necessary and sufficient
conditions for survival of $Z$.
\begin{thm}\label{thm:BPwM}
    Consider the branching process with migration $Z=(Z_n)_{n \geq 0}$ as above.
    Then
    \begin{itemize}
        \item
            If $\mu>1$, then $Z$ a.s. survives.
        \item
            If $\mu<1$ then $Z$ a.s. dies out.
        \item
            Assume $\mu=1$, then $Z$ dies out a.s. if and only if $\theta  = \frac{2\rho}{\nu} \leq 1$.
    \end{itemize}
\end{thm}

\begin{proof}
    Note that the process $Z=(Z_n)_{n \geq 0}$ is a Markov chain on $\N_0$ with the
    step distribution
    \[
        U(x) =
                \begin{cases}
                    \max\left\{\sum_{i=1}^x \xi_i^{(1)} + \eta,0 \right\} & \mbox{ if } x > 0 \\
                    0 & \mbox{ if } x = 0.
                \end{cases}
    \]
    Our Theorem \ref{thm:Fundamental} is formulated for irreducible
    chains, but we can simply change $U(0)$ to be, say, 1, and replace
    ``dies out'' with ``reaches 0'' and we are back in the irreducible
    case. % and aperiodic
    %as the support of $\xi$ is $\N_0$.
    %(Note that it is actually irreducible everywhere except in $0$,
    %but since the Markov chain stops when it reaches $0$, then it can be assumed
    %to be irreducible also in $0$ without loss of generality.)
    We now apply Theorem~\ref{thm:Fundamental} to the process $Z$. Note that:
    \begin{enumerate}
    \item
        The sum $\sum_{i=1}^x \xi_i^{(1)}+\eta$ is concentrated around its mean,
        which follows from Hoeffding's type inequality for random
        variables with exponential tails.
        In particular, $\frac{U(x)}{x}$ is concentrated around $\mu$.
    \item
        This shows that the effect of taking the maximum with zero
        is negligible. Indeed, for large values of $x$ we have
        \begin{align*}
            \Pr[U(x) = 0]
            & \leq
            \Pr \left[ \sum_{i=1}^x \xi_i^{(1)} > \mu x/2 \text{ and } \eta < - \mu x/2 \right]
            +\Pr \left[ \sum_{i=1}^x \xi_i^{(1)} < \mu x/2\right]  \\
            & \leq \Pr \left[ \eta < - \mu x/2 \right]
            +\Pr \left[ \sum_{i=1}^x \xi_i^{(1)} < \mu x/2\right]  \\
            & \leq \exp(-cx),
        \end{align*}
        for some constant $c$ which depends on $\xi$ and $\eta$ but not on $x$.
        Therefore $|\E[U(x)] - \mu x + \rho|\le  \exp(-c'x)$
        for some constant $c'>0$ that depends on $\xi$ and $\eta$ but not on $x$.
    \item
        By independence of $\xi_i^{(n)}$'s and $\eta_{(m)}$ we have
        \begin{align*}
            \E[(U(x)- \mu x)^2] &= \E[(\sum_{i=1}^{x}(\xi_i^{(1)} -
              \mu) + \eta^2] + O(e^{-cx})\\
            &= \sum_{i=1}^{x}\E[(\xi -\mu)^2] + \E[\eta^2] +
            O(e^{-cx})\\
            & = \nu x + \E[\eta^2] + O(e^{-cx}),
        \end{align*}
        and hence $\frac{\E[U(x)-x]^2}{x} = \nu + \frac{\E[\eta^2]}{x}+O(e^{-cx})$.
    \end{enumerate}
    Therefore, by applying Theorem~\ref{thm:Fundamental} we get the desired conclusion.
\end{proof}

%%%%%%%%%%%%%%%%%%%%%%%%%%%%%%%%%%%%%%%%%%%%%%%%%%%%%%%%%%%%%%%%%%%%%%%%%%%%%%%
\section{Open problems}\label{sec:Open Problems}
%%%%%%%%%%%%%%%%%%%%%%%%%%%%%%%%%%%%%%%%%%%%%%%%%%%%%%%%%%%%%%%%%%%%%%%%%%%%%%%

\begin{enumerate}
\item
For ERW with periodic environments, %find an explicit condition in terms of the period for positive speed of the walk.
compute the speed in terms of the period.
\item
Find an identically piled (uniformly) elliptic cookie environments
so that $\mu= \theta = 1$ and the walk is right transient. Note that by
Theorem~\ref{thm:Fundamental} it is enough to find an environment so that
$\theta(x)-1$ is eventually larger than $\frac{2}{\ln(x)} + \alpha(x)\cdot x^{-\half}$
for some $\alpha(x)$ such that $\alpha(x)\nu(x) \to+\infty$.
\end{enumerate}

%%%%%%%%%%%%%%%%%%%%%%%%%%%%%%%%%%%%%%%%%%%%%%%%%%%%%%%%%%%%%%%%%%%%%%%%%%%%%%%
\section*{Acknowledgments}
%%%%%%%%%%%%%%%%%%%%%%%%%%%%%%%%%%%%%%%%%%%%%%%%%%%%%%%%%%%%%%%%%%%%%%%%%%%%%%%

%We would like to thank the anonymous referees for their remarks and suggestions
%that helped us improving the paper, both in style and in content.
T.O.~would like to thank Ofer~Zeitouni for his enlightening ideas and constant encouragement in many hours of discussions.
T.O.~thanks also Vitali Wachtel for the reference to Lamperti's
work~\cite{Lamperti}.
We are also grateful to the anonymous referee who pointed out to us the paper
of Menshikov, Asymonth, and Iasnogorodski~\cite{zerodrift}. We thank Nick Travers for spotting a mistake in
the statement of Theorem \ref{thm:Fundamental}.
We thank Itai~Benjamini for useful discussions.
G.K.~and T.O.~are partially supported by the Israel Science Foundation.
I.S.~is supported by ERC grant number 239985.
\bibliography{BibERW}
\bibliographystyle{plain}

%%%%%%%%%%%%%%%%%%%%%%%%%%%%%%%%%%%%%%%%%%%%%%%%%%%%%%%%%%%%%%%%%%%%%%%%%%%%%%%
\appendix
%%%%%%%%%%%%%%%%%%%%%%%%%%%%%%%%%%%%%%%%%%%%%%%%%%%%%%%%%%%%%%%%%%%%%%%%%%%%%%%

%%%%%%%%%%%%%%%%%%%%%%%%%%%%%%%%%%%%%%%%%%%%%%%%%%%%%%%%%%%%%%%%%%%%%%%%%%%%%%%
\section{Survival of irreducible % aperiodic
Markov chains on \texorpdfstring{$\N_0$}{the integers}}\label{sec:mc survival criterion}
%%%%%%%%%%%%%%%%%%%%%%%%%%%%%%%%%%%%%%%%%%%%%%%%%%%%%%%%%%%%%%%%%%%%%%%%%%%%%%%

%This hides the \subsection* commands from the table of contents.
\addtocontents{toc}{\protect\setcounter{tocdepth}{1}}

In this appendix we prove our criterion for transience of
Markov chains on $\N_0$ stated in Theorem~\ref{thm:Fundamental}.
Recall that we denote by
$Z = (Z_n)_{n \geq 0}$ an irreducible % and aperiodic
discrete time Markov chain on $\N_0$ starting at $Z_0 = 1$, and
that we denote by $U=(U(x))_{x \geq 0}$ its step distribution. Recall
also the asymptotic mean $\mu$, the drift
$\rho(x)$, the diffusion constant $\nu(x)$ and the ratio $\theta(x)$
defined just before Theorem~\ref{thm:Fundamental} (page~\pageref{thm:Fundamental}).
The proof of Theorem~\ref{thm:Fundamental} relies on the classical
approach of \emph{Lyapunov functions}. Theorems~2.1 and~2.2 of
Lamperti~\cite{Lamperti} will serve as a convenient
reference. The following theorem is an immediate corollary of them.

\begin{thm}\label{thm:LyapunovFosterCriterion}
Let $Z$ be an irreducible %and aperiodic
discrete time Markov chain on $\N_0$, with step distribution $U = (U(x))_{x \geq 0}$.
That is $\Pr[Z_{n+1}=y|Z_n=x] = \Pr[U(x)=y]$ for all $n \geq 0$.
Then
    \begin{enumerate}
    \item
        $Z$ is recurrent whenever there is some function $V : \N_0 \to (0,\infty)$ such that
        $\lim_{x \to \infty}V(x) = \infty$ and $\E[V(U(x))] \leq V(x)$ for all
        sufficiently large values of $x$.
    \item
        $Z$ is transient whenever there is some function $V: \N_0 \to (0,\infty)$
        %strictly positive and a finite set $A \subseteq \N_0$
        such that $\lim_{x \to \infty}V(x) = 0$
        and $\E[V(U(x))] \leq V(x)$ for all %$x\in \N_0 \setminus A$.
        sufficiently large values of $x$.
    \end{enumerate}
%    Here $Z$ is said to be \emph{recurrent} if $\Pr[Z_n=0 \mbox{ for some } n] = 1$
%    and is called \emph{transient} otherwise.
\end{thm}

A function $V$ satisfying one of the two possibilities in Theorem~\ref{thm:LyapunovFosterCriterion}
is called \emph{Lyapunov function} for the Markov chain defined by $U$.

We start our proof with the two simple cases of $\mu<1$ and $\mu>1$.

\subsection*{The case \texorpdfstring{$\boldsymbol{\mu<1}$}{mu
    smaller than 1}:}
We apply Theorem~\ref{thm:LyapunovFosterCriterion} on $U$ with Lyapunov function $V(x)=x$.
We claim that for all $x$ sufficiently large it holds that $\E[U(x)] \leq x$.
Indeed, $\E[U(x)] \leq \mu x + o(x) < 0$ for all sufficiently large $x$
since $\mu<1$. We are done since $V(x)\to\infty$ as $x\to\infty$.

\subsection*{The case \texorpdfstring{$\boldsymbol{\mu>1}$}{mu greater
    than 1}:}
Define $V(x):=\frac{1}{x+1}$. We claim that
for all sufficiently large $x$ we have $\E[\frac{1}{U(x)+1}] \leq \frac{1}{x+1}$.
Indeed, using first order Taylor expansion applied to the function
$f(U) = \frac{1}{U+1}$ around $x$ we have
\[
    \E \left[\frac{1}{U(x)+1} \right] = \frac{1}{x+1} - \E \left[ \frac{1}{(1+\xi)^2} \cdot (U(x)-x) \right]
\]
for some $\xi$ lying between $x$ and $U(x)$. By the concentration of $U$
for $x$ sufficiently large we have
$\frac{1+\mu}{2} \leq \frac{U(x)}{x} \leq 2\mu$ with high probability,
in which case the expression in the expectation is
$\frac{U(x)-x}{(1+\xi)^2} \geq \frac{c}{x}$ for some constant $c$ that depends only on $\mu$.
Note that either way the expression in the expectation is at least $-x$.
Therefore, if we denote
$p_{\mu} = \Pr \left[\frac{U(x)}{x} \in [\frac{1+\mu}{2}, 2\mu ] \right]$,
then
\[
    \E \left[\frac{1}{U(x)+1} \right]
    \leq \frac{1}{x+1} -
    (p_{\mu}  \cdot \frac{c}{x} + (1 - p_{\mu}) \cdot (-x))
    \leq \frac{1}{x+1},
\]
where the last inequality follows from the concentration of $U(x)$, which
implies that $p_{\mu}$ is exponentially close to $1$.
This completes the proof of the case $\mu>1$.

\subsection*{The case \texorpdfstring{$\boldsymbol{\mu=1}$}{mu equal
    to 1}:}
%We now turn to the proof of the case of $\mu=1$.
The proof for the case $\mu=1$ uses again Theorem~\ref{thm:LyapunovFosterCriterion}
with an appropriate Lyapunov function. For the recurrence case
the function we will use is $V(x)=\ln\ln(x)\to\infty$,
and for the transience we will use $V(x)=\ln^{-1}(x)\to 0$.
In both cases we use Taylor expansion of $V$ around $x$ to prove
that $V(U(x))$ satisfies the super-martingale property, namely,
that $\E[V(U(x))] \leq V(x)$ for all $x$ sufficiently large.

\subsection*{The case \texorpdfstring{$\boldsymbol{\theta(x) - 1 \ll
      \frac{1}{\ln(x)}}$}{theta smaller than 1}:}
This case is summarized in the following claim.
\begin{claim}
    Suppose that $\theta(x) < 1 + \frac{1}{\ln(x)} - \alpha(x)\cdot x^{-\half}$
    for all sufficiently large $x \in \N_0$, where $\alpha(x)$ is such that $\alpha(x)\nu(x) \to+\infty$.
    Then $\Pr[Z_n=0 \mbox{ for some } n] = 1$.
\end{claim}

\begin{proof}
    We define our Lyapunov function to be $V(x) = \ln\ln(x)$.%
    \footnote{Note that $V(x)$ is not defined properly for $x \leq e$.
    We overcome this by defining $V$ in a range slightly larger that $[0,e]$
    arbitrarily, while making sure that $V$ is smooth and positive.}
    We claim that for all $x$ sufficiently large it holds that
    $\E[\ln\ln U(x)] \leq \ln\ln x$,
    which by Theorem~\ref{thm:LyapunovFosterCriterion} implies the claim.

    We state the first three derivatives of $V$, which hold for all sufficiently large values of $x$.
    \begin{eqnarray*}
        V'(x) &=& \frac{1}{x \ln(x)}        \\
        V^{(2)}(x) &=& - \frac{1}{x^2 \ln(x)} - \frac{1}{x^2 \ln^2(x)}        \\
        V^{(3)}(x) &=& \frac{2}{x^3 \ln(x)} + \frac{3}{x^3 \ln^2(x)} + \frac{2}{x^3 \ln^3(x)}.
        %V^{(4)}(x) &=& - \frac{6}{x^4 \ln(x)} - \frac{11}{x^4 \ln^2(x)} - \frac{9}{x^4 \ln^3(x)} - \frac{3}{x^4 \ln^4(x)}
    \end{eqnarray*}
%    Note that if $x$ is sufficiently large, then $V^{(3)}(x) \le C/x^3 \ln(x)$.
%    and $V^{(4)}(x) \leq 0$.
    Using $3^\textrm{rd}$ order Taylor expansion of $V$ around $x$ with Cauchy remainder we have for all large enough $x$
    \begin{align*}
        \E[\ln\ln(U)]
            & = \ln\ln(x) + V'(x) \cdot \E[U-x]
                + \frac{1}{2!} V^{(2)}(x) \cdot \E[(U-x)^2] \\
                &\quad  + \frac{1}{3!} \E[V^{(3)}(\xi) \cdot (U-x)^3],
%                + \frac{1}{4!} \E[V^{(4)}(C) \cdot (U-x)^4]
    \end{align*}
    where $\xi$ is some random value between $x$ and $U$. It is easy
    to see that the exponential concentration of $U$ implies that the
    remainder is $O(x^{-3/2}\ln^{-1}(x))$ (regardless of how $V$ is
    defined for small values of $x$).
    Inserting the definitions of $\rho(x)$ and $\nu(x)$ we have
    \[
        \E[\ln\ln(U)]
                = \ln\ln(x) + \frac{\rho(x)}{x \ln(x)}
                - \frac{x \nu(x)}{2 x^2 \ln(x)}
                - \frac{x \nu(x)}{2 x^2 \ln^2(x)} +
                O\left(\frac{x^{-3/2}}{\ln x}\right).
%                &&  + \frac{1}{3!} V^{(3)}(x) \cdot \E[(U-x)^3]
%                + \frac{1}{4!} \E[V^{(4)}(C) \cdot (U-x)^4].
    \]
%    Using the fact that $\E[V^{(4)}(C) \cdot (U-x)^4] \leq 0$
%    in order to prove that $\E[\ln\ln U(x)] \leq \ln\ln x$ it is enough to show that
%    \[
%        \frac{\rho(x)}{x \ln(x)} \leq
%        \frac{\nu(x)}{2 x \ln(x)}
%        + \frac{\nu(x)}{2 x \ln^2(x)}
%        - \frac{1}{3!} V^{(3)}(x) \cdot \E[(U-x)^3].
%    \]
    Multiplying %both sides of the inequality
    by $\smash{\frac{2 x
      \ln(x)}{\nu(x)}}$, and recalling that
    $\smash{\theta(x) = \frac{2\rho(x)}{\nu(x)}}$ we see that it is enough to
    show that %this is equivalent to showing that
    \[
        \theta(x) \leq
            1 + \frac{1}{\ln(x)} + O(x^{-1/2}).
%            - \frac{x \ln(x)}{3\nu(x)} \cdot V^{(3)}(x) \cdot \E[(U-x)^3]
    \]
%    By the concentration of $U$ we can apply
%    Proposition~\ref{prop:centralized moments} with $m=3$ and conclude that
%    $\E[(U-x)^3] = O(x^{\frac{3}{2}})$.
%    Therefore, using the fact that $V^{(3)}(x) = \Theta(\frac{1}{x^3 \ln(x)})$
%    the third term on left hand side of the last inequality is at most $O(x^{-\half})$.
    Since our assumption of $\theta$ was that $\theta(x) < 1 + \frac{1}{\ln(x)} - \alpha(x) x^{-\half}$, and $\alpha(x)\nu(x) \to+\infty$,
    the required inequality holds for all $x$ sufficiently large,
    and therefore for such values of $x$ we have
    $\E[\ln\ln U(x)] \leq \ln\ln x$, as required.
\end{proof}

\subsection*{The case \texorpdfstring{$\boldsymbol{\theta(x) - 1 \gg
      \frac{2}{\ln(x)}}$}{theta bigger than 1}:}
This case is summarized in the following claim.
\begin{claim}
    Suppose that $\theta(x) > 1 + \frac{2}{\ln(x)} + \alpha(x)\cdot x^{-\half}$
    for all sufficiently large $x \in \N_0$, where $\alpha(x)$ is such that $\alpha(x)\nu(x) \to+\infty$.
    Then $\Pr[Z_n > 0 \mbox{ for all } n] > 0$.
\end{claim}
\begin{proof}
    We define our Lyapunov function to be $V(x) = \ln^{-1}(x)$.%
    \footnote{Just like in the previous case $V(x)$ is not defined in $x=1$,
    and it is not positive for $x<1$.
    Again, we overcome this by defining $V$ in the interval $[0,2]$
    arbitrarily, while making sure that $V$ is smooth and positive.}
    We claim that for all $x$ sufficiently large it holds that
    $\E[\ln^{-1}(U(x))] \leq \ln^{-1}(x)$,
    which by Theorem~\ref{thm:LyapunovFosterCriterion} implies the claim.

    We state the first three derivatives of $V$, which hold for all sufficiently large values of $z$.
    \begin{eqnarray*}
        V'(x) &=& - \frac{1}{x \ln^2(x)}        \\
        V^{(2)}(x) &=& \frac{1}{x^2 \ln^2(x)} + \frac{2}{x^2 \ln^3(x)}      \\
        V^{(3)}(x) &=& - \frac{2}{x^3 \ln^2(x)} - \frac{6}{x^3 \ln^3(x)} - \frac{6}{x^3 \ln^4(x)}.
    \end{eqnarray*}
    Using $3^\textrm{rd}$ order Taylor expansion of $V$ around $x$ with Cauchy remainder we have
    \begin{align*}
        \E[\ln^{-1}(U)]
            & = \ln^{-1}(x) + V'(x) \cdot \E[U-x]
                + \frac{1}{2!} V^{(2)}(x) \cdot \E[(U-x)^2]\\
                &\quad + \frac{1}{3!} \E[V^{(3)}(\xi) \cdot (U-x)^3]
    \end{align*}
    for some random $\xi$ between $x$ and $U$. As before the
    exponential concentration of $U$ gives that the error is $O(x^{-3/2}\ln^{-2}(x))$.
    %The following claim easily follows from the concentration of $U$,
    %and its proof is omitted.
    %\begin{claim}
    %    For all $x$ sufficiently large and for any $C$ between $x$ and $U$
    %    it holds that
    %    \[
    %        \E[V^{(3)}(C) \cdot  (U-x)^3] \leq
    %        O \left( \frac{1}{x^{\frac{3}{2}} \ln^2(x)} \right).
    %    \]
    %\end{claim}
    By the definition of $\rho(x)$ and $\nu(x)$ we have
\[
        \E[\ln^{-1}(U)]
          = \ln^{-1}(x) - \frac{\rho(x)}{x \ln^2(x)}
            + \frac{x\nu(x)}{2x^2 \ln^2(x)}
            + \frac{x\nu(x)}{x^2 \ln^3(x)}
            +O\left(\frac{x^{-3/2}}{\ln^2(x)}\right).
\]
%
%    \begin{eqnarray*}
%        \E[\ln^{-1}(U)]
%                & = & \ln^{-1}(x) - \frac{\rho(x)}{x \ln^2(x)} \\
%                && + \frac{x\nu(x)}{2x^2 \ln^2(x)}
%                    + \frac{x\nu(x)}{x^2 \ln^3(x)} \\
%                &&  + \frac{1}{3!} \cdot \E[V^{(3)}(C) \cdot (U-x)^3]
%    \end{eqnarray*}
    Therefore, in order to prove that $\E[\ln^{-1}(U(x))] \leq \ln^{-1}(x)$ it is enough to show that
    \[
        \frac{\rho(x)}{x \ln^2(x)} \geq
        \frac{\nu(x)}{2 x \ln^2(x)}
        + \frac{\nu(x)}{x \ln^3(x)}
          +O\left(\frac{x^{-3/2}}{\ln^2(x)}\right).
          %+ \frac{1}{3!} \E[V^{(3)}(C) \cdot (U-x)^3]
    \]
    Multiplying both sides of the inequality by $\frac{2 x \ln^2(x)}{\nu(x)}$, and substituting
    $\theta(x) = \frac{2\rho(x)}{\nu(x)}$ this is equivalent to showing that
    \[
        \theta(x) \geq
            1 + \frac{2}{\ln(x)}
             + O(x^{-1/2}).%\frac{x \ln^2(x)}{3\nu(x)} \cdot \E[V^{(3)}(C) \cdot (U-x)^3].
    \]
%    By plugging $\E[V^{(3)}(C) \cdot  (U-x)^3] \leq O \left( \frac{1}{x^{\frac{3}{2}} \ln^2(x)} \right)$
%    from the last claim, we get that the third term on the right hand side is at most $O(x^{-\half})$.
    Therefore, if $\theta(x) > 1 + \frac{2}{\ln(x)} + \alpha(x)x^{-\half}$ for some $\alpha(x)$ such that $\alpha(x)\nu(x) \to+\infty$, then
    the above inequality holds for all large enough $x$. The claim,
    and hence also Theorem~\ref{thm:Fundamental}, follow.
\end{proof}

\end{document}